\documentclass[12pt,psfig,reqno]{amsart}
\usepackage{}
\usepackage{epsfig}
\usepackage{graphicx}
\usepackage{mathrsfs}
\usepackage{verbatim}
\usepackage{amssymb}
\usepackage{latexsym}
\usepackage{amsmath}
\usepackage{amsfonts}
\usepackage{amsbsy}
\usepackage{amscd}
\usepackage[usenames]{color} 
\usepackage{hyperref}
\usepackage{subfigure}
\usepackage{tikz}
\usepackage[utf8]{inputenc}
 \numberwithin{equation}{section}

\newtheorem{thm}{Theorem}[section]

\newtheorem{coro}[thm]{Corollary}

\newtheorem{lem}[thm]{Lemma}

\newtheorem{prop}[thm]{Proposition}
\newtheorem{exam}[thm]{Example}

\theoremstyle{definition}

\newtheorem{defi}[thm]{Definition}
\newtheorem{rema}[thm]{Remark}
\numberwithin{equation}{section}

\begin{document}
\baselineskip=17pt
\title[Separation conditions for IFSs on Riemannian manifolds]
{Separation conditions for iterated function systems with overlaps on Riemannian manifolds}
\date{\today}
\author[S.-M. Ngai]{Sze-Man Ngai}
\address{Key Laboratory of High Performance Computing and Stochastic Information
		Processing (HPCSIP) (Ministry of Education of China), College of Mathematics and Statistics, Hunan Normal University,
	Changsha, Hunan 410081, China, and Department of Mathematical Sciences\\ Georgia Southern
	University\\ Statesboro, GA 30460-8093, USA.}
\email{smngai@georgiasouthern.edu}
\author[Y.Y. Xu]{Yangyang Xu*}
\address{School of Mathematical Sciences, Zhejiang University,
	Hangzhou, Zhejiang 310027, China.} \email{yyxumath@163.com}

\thanks{* Corresponding author.}

\begin{abstract}
We formulate the weak separation condition and the finite type condition for conformal iterated function systems on Riemannian manifolds with nonnegative Ricci curvature, and generalize the main theorems by Lau \textit{et al.} in [Monatsch. Math. 156 (2009), 325-355]. We also obtain a formula for the Hausdorff dimension of a self-similar set defined by an iterated function system satisfying the finite type condition, generalizing a corresponding result by Jin-Yau [Comm. Anal. Geom. 13 (2005), 821--843] and Lau-Ngai [Adv. Math. 208 (2007), 647-671] on Euclidean spaces.
Moreover, we obtain a formula for the Hausdorff dimension of a graph self-similar set generated by a graph-directed iterated function system satisfying the graph finite type condition, extending a result by Ngai \textit{et al.} in [Nonlinearity 23 (2010), 2333--2350].

\end{abstract}

\maketitle

\section{Introduction}\label{S:0}
The weak separation condition (WSC) was introduced by Lau and Ngai \cite{Lau-Ngai_1999} to study the multifractal formalism of self-similar measures defined by iterated function systems of contractive similitudes with overlaps. Although strictly weaker than the well-known open set condition (OSC), the weak separation condition leads to many interesting results, such as the validity of the multifractal formalism (see, e.g. \cite{Lau-Ngai_1999,Feng_2003,Feng_2005,Feng-Lau_2009, Shmerkin_2005,Ye_2005}),  equality of Hausdorff, box, and packing dimensions of self-conformal sets and the computation of these dimensions \cite{Deng-Ngai_2011,Ferrari-Panzone_2011,Lau-Ngai-Wang_2009}, and conditions on absolute continuity of self-similar and self-conformal measures \cite{Lau-Ngai-Rao_2001,Lau-Wang_2004,Lau-Ngai-Wang_2009}. Equivalent forms of the weak separation conditions have also been studied extensive (see, e.g. \cite{Zerner_1996, Lau-Ngai-Wang_2009}).

The finite type condition (FTC) was introduced by Ngai and Wang \cite{Ngai-Wang_2001} to calculate the Hausdorff dimension of self-similar sets with overlaps. It was generalized independently by Jin-Yau \cite{Jin-Yau_2005} and Lau-Ngai \cite{Lau-Ngai_2007} to include (OSC). Lau and Ngai proved that (FTC) implies (WSC) for IFSs of contractive similitudes. This result was generalized by Lau \textit{et al.} \cite{Lau-Ngai-Wang_2009} to conformal iterated function systems (CIFSs).

In 2009, Lau \textit{et al.} \cite{Lau-Ngai-Wang_2009} formulated (WSC) for conformal iterated function systems on $\mathbb{R}^{n}$, and proved the equality of the Hausdorff, box and packing dimensions of the associated self-conformal sets. They also studied the absolute continuity of the associated self-conformal measures. The first goal of this paper is to extend results in \cite{Lau-Ngai-Wang_2009} to Riemannian manifolds with nonnegative Ricci curvature. Our second goal is to generalize the method of computing the Haudorff dimension of self-similar sets in \cite{Ngai-Wang_2001,Jin-Yau_2005,Lau-Ngai_2007} to Riemannian manifolds that are locally Euclidean.

Let $M$ be a complete $n$-dimensional smooth Riemannian manifold.
Assume that $U\subset M$ is open and connected, and $W\subset U$ is a compact set with $\overline{W^{\circ}}=W$, where $\overline{W^{\circ}}$ is the closure of the interior of $W$. Assume that $\{S_{i}\}_{i=1}^{N}$ is a conformal iterated function system (CIFS) on $U$ defined as in Section \ref{S:1}. Then there exists a unique nonempty compact set $K\subset W$, called the \textit{attractor} or \textit{self-conformal set}, satisfying $K=\bigcup_{i=1}^{N}S_{i}(K)$ (see \cite{Hutchinson_1981}).

Given a probability vector $(p_{1},\dots,p_{N})$, i.e., $p_{i}>0$ for any $i\in\{1,\dots,N\}$ and $\sum_{i=1}^{N}p_{i}=1$, there exists a unique Borel probability measure $\mu$, called the \textit{self-conformal measure}, such that $\mu=\sum_{i=1}^{N}p_{i}\mu\circ S_{i}^{-1}$ and $K={\rm supp}(\mu)$ (see \cite{Hutchinson_1981}).

Let $\Sigma:=\{1,\dots,N\}$, where $N\in\mathbb{N}$ and $N\geq2$. Let $\Sigma^{\ast}:=\bigcup_{k\geq1}\Sigma^{k}$, where $k\in\mathbb{N}$. For $\mathbf{u}=(u_1,\dots,u_k)\in\Sigma^{k}$, let $\mathbf{u}^{-}:=(u_{1},\dots,u_{k-1})$ and write $S_{\mathbf{u}}:=S_{u_1}\circ\cdots\circ S_{u_k}$, $p_{\mathbf{u}}:=p_{u_1}\cdots p_{u_k}$. Define
$$r_{\mathbf{u}}:=\inf_{x\in W}|\det S'_{\mathbf{u}}(x)|^{\frac{1}{n}},~r:=\min_{1\leq i\leq N}r_{i},~R_{\mathbf{u}}:=\sup_{x\in W}|\det S'_{\mathbf{u}}(x)|^{\frac{1}{n}},~R:=\max_{1\leq i\leq N}R_{i}.$$
If $S=S_{\mathbf{u}}$ for some $\mathbf{u}\in\Sigma^{\ast}$, we let $R_{S}=R_{\mathbf{u}}$. For $0<b\leq1$, let
$$\mathcal{W}_{b}:=\{\mathbf{u}=(u_1,\dots,u_k):R_{\mathbf{u}}\leq b<R_{\mathbf{u}^{-}}\}\quad\text{and}\quad\mathcal{A}_{b}:=\{S_{\mathbf{u}}:\mathbf{u}\in\mathcal{W}_{b}\}.$$
Denote the cardinality of a set $E$ by $\#E$. We say that $\{S_{i}\}_{i=1}^{N}$ satisfies \textit{the weak separation condition} (WSC) if there exists a constant $\gamma\in\mathbb{N}$ and a subset $D\subset W$, with $D^{\circ}\neq\emptyset$, such that for any $0<b\leq1$ and $x\in W$,
$$\#\{S\in\mathcal{A}_{b}:x\in S(D)\}\leq\gamma.$$
Remark that if the open set condition (OSC) holds, we can take $D$ to be an OSC set and let $\gamma=1$ to show that (WSC) also holds.
For any $a>0$ and any bounded subsets $D\subset W$ and $A\subset M$, denote the diameter of $A$ by $|A|$, and let
$$\mathcal{A}_{a,A,D}:=\{S\in\mathcal{A}_{a|A|}:S(D)\cap A\neq\emptyset\},\quad\gamma_{a,A}:=\sup_{A\subset M}\#\mathcal{A}_{a,A,D}.$$
For $S\in\mathcal{A}_{b}$, let $p_{S}:=\sum\{p_{\mathbf{u}}:S_{\mathbf{u}}=S,\mathbf{u}\in\mathcal{A}_{b}\}$.

Theorems \ref{thm(0.1)}--\ref{thm(0.4)} generalize analogous results in \cite{Lau-Ngai-Wang_2009}. In our proofs, the Lebesgue measure in \cite{Lau-Ngai-Wang_2009} is changed to the more complicated Riemannian volume measure; properties such as the volume doubling property, need not hold. We assume that $M$ is a complete Riemannian  manifold with nonnegative Ricci curvature. Under this assumption, the Bishop-Gromov comparison theorem implies that $M$ is a \textit{doubling space} (see, e.g. \cite{Baudoin_2011,Berger_2003}), i.e., any $2r$-ball in $M$ can be covered by a finite union of a bounded number of $r$-balls,  a property that obviously holds on $\mathbb{R}^{n}$. Another complication arises on manifolds; unlike Euclidean spaces, it is not easy to calculate the volume of a ball in Riemannian manifolds.  For Riemannian manifolds with nonnegative Ricci curvature, we use the Bishop-Gromov inequality  (see Lemma \ref{lem(1.00)}), which says that the Riemannian volume of a ball can be controlled by a ball in $\mathbb{R}^{n}$ with the same radius. This  is crucial in the proof of Theorem \ref{thm(0.1)}.

For a set $K\subset M$, let $\dim_{{\rm H}}(K),\dim_{{\rm B}}(K)$ and $\dim_{{\rm P}}(K)$ be the Hausdorff, box and packing dimensions, respectively. Let $\mathcal{H}^{\alpha}(K)$ and $\mathcal{P}^{\alpha}(K)$ be the Hausdorff and packing measures of $K$, respectively.

\begin{thm}\label{thm(0.1)}
Let $M$ be a complete $n$-dimensional smooth orientable Riemannian manifold with non-negative Ricci curvature, and let $U\subset M$ be open and connect. Assume that $\{S_{i}\}_{i=1}^{N}$ is a CIFS on $U$ satisfying (WSC), and $K$ is the associated attractor. Then $\alpha:=\dim_{{\rm H}}(K)=\dim_{{\rm B}}(K)=\dim_{{\rm P}}(K)$ and
$0<\mathcal{H}^{\alpha}(K)\leq\mathcal{P}^{\alpha}(K)<\infty$.
\end{thm}

Lau \textit{et al.} \cite{Lau-Ngai-Rao_2001} formulated a sufficient condition for a self-similar measure defined by an IFS satisfying (WSC) to be singular. Later,  Lau and Wang \cite{Lau-Wang_2004} established the necessary perfected the result on absolute continuity in \cite{Lau-Ngai-Rao_2001}. We extend these results to manifolds.

\begin{thm}\label{thm(0.2)}
Assume the same hypotheses as in Theorem \ref{thm(0.1)}. Let $K$ be the attractor
with $\dim_{{\rm H}}(K)=\alpha$. Then a self-conformal measure $\mu$ defined by $\{S_{i}\}_{i=1}^{N}$ is singular with respect to $\mathcal{H}^{\alpha}|_{K}$ if and only if there exist $0<b\leq1$ and $S\in\mathcal{A}_{b}$ such that $p_{S}>R_{S}^{\alpha}$.
\end{thm}

\begin{thm}\label{thm(0.3)}
Assume the same hypotheses as in Theorem \ref{thm(0.1)}. If the self-conformal measure $\mu$ is absolutely continuous with respect to $\mathcal{H}^{\alpha}|_{K}$, then the
Radon-Nikodym derivative of $\mu$ is bounded.
\end{thm}

In the proof of Theorem \ref{thm(0.3)}, we use an analogue of the Lebesgue density theorem in metric spaces (see, e.g. \cite[Theorem 2.9.8]{Federer_1969} and \cite[Lemma 2.1(i)]{Bedferd-Fisher_1992}), applying to the collection of Borel sets that forms a Vitali relation (see \cite{Federer_1969} and a brief summary in Section \ref{S:2}). By \cite[Definition 2.8.9 and Theorem 2.8.18]{Federer_1969}, we have a collection of open balls in Riemannian manifolds that forms a Vitali relation.

 We refer the reader to Section \ref{S:3} for the definition of (FTC).

\begin{thm}\label{thm(0.4)}
Assume the same hypotheses as in Theorem \ref{thm(0.1)}, and let $W\subset U$ be a compact set with $\overline{W^{\circ}}=W$. If $\{S_{i}\}_{i=1}^{N}$ is a CIFS on $U$ satisfying (FTC) on some open sets $\Omega\subset W$,
then $\{S_{i}\}_{i=1}^{N}$ satisfies (WSC).
\end{thm}

Denote the Riemannian distance in $M$ by $d(\cdot,\cdot)$. Let $W\subset M$ be a compact set.
We say that $\{S_{i}\}_{i=1}^{N}$ is an \textit{IFS of contractions} on $W$ if for any $i\in\{1,\dots,N\}$, there exists $\rho_{i}\in(0,1)$ such that for any $x,y\in W$,
\begin{equation}\label{eq:IFS_contraction}
d(S_{i}(x),S_{i}(y))\le \rho_{i}d(x,y).
\end{equation}
If equality in \eqref{eq:IFS_contraction} holds for all $i\in\{1,\dots,N\}$ and all $x,y\in W$, then
say that $\{S_{i}\}_{i=1}^{N}$ is an \textit{IFS of contractive similitudes} on $W$ and call $\rho_i$ the \textit{contraction ratio} of $S_{i}$.

We say that a Riemannian manifold $M$ is \textit{locally Euclidean} if every point of $M$ has a neighborhood which is isometric to an open subset of a Euclidean space (see e.g. \cite{Kobayashi-Nomizu}). By \cite[Lemma 2 of Theorem 3.6]{Kobayashi-Nomizu}, contractive similitudes only exist in Riemannian manifolds that are locally Euclidean. In Section \ref{S:4}, we obtain the following formula for computing the Hausdorff dimension formula of self-similar sets defined by a finite type IFS of contractive similitudes on a locally Euclidean Riemannian manifold, extending a result in \cite{Jin-Yau_2005} and \cite{Lau-Ngai_2007} to locally Euclidean Riemannian manifolds (see Theorem \ref{thm(0.5)}).

\begin{thm}\label{thm(0.5)}
Let $M$ be a complete $n$-dimensional smooth orientable Riemannian manifold that is locally Euclidean, $W\subseteq M$ be a compact subset, and $\{S_{i}\}_{i=1}^{N}$ be an IFS of contractive similitudes on $W$ with attractor $K$. Let $\lambda_{\alpha}$ be the spectral radius of the associated weighted incidence matrix $A_{\alpha}$. If $\{S_{i}\}_{i=1}^{N}$ satisfies (FTC), then $\dim_{{\rm H}}(K)=\alpha$, where $\alpha$ is the unique number such that $\lambda_{\alpha}=1$.
\end{thm}

In order to compute the Hausdorff dimension of certain attractors on Riemannian manifolds, it is necessary to study graph iterated function systems (see Example~\ref{exam(5.4)}). We define graph-directed iterated function systems (GIFSs) and graph finite type condition (GFTC) on Riemannian manifolds in Section \ref{S:41}. We obtain the following result for computing the Hausdorff dimension of graph self-similar sets (see Theorem \ref{thm(41.1)}), extending a result in \cite{Ngai-Wang-Dong_2010}.

\begin{thm}\label{thm(41.1)}
Let $M$ be a complete $n$-dimensional smooth orientable Riemannian manifold that is locally Euclidean. Assume that $G=(V,E)$ is a GIFS defined on $M$ satisfying (GFTC), and $K$ is the associated graph self-similar set. Let $\lambda_{\alpha}$ be the spectral radius of the associated weighted incidence matrix $A_{\alpha}$. Then $\dim_{{\rm H}}(K)=\alpha$, where $\alpha$ is the unique number such that $\lambda_{\alpha}=1$.
\end{thm}

This paper is organized as follows. Section \ref{S:1} introduces the definition of CIFSs, some properties of (WSC), and gives the proof of Theorem \ref{thm(0.1)}. In Section \ref{S:2}, we study the absolute continuity of self-conformal measures on Riemannian manifolds and prove Theorems \ref{thm(0.2)} and \ref{thm(0.3)}. Section \ref{S:3} is devoted to the proof of Theorem \ref{thm(0.4)}. In Section \ref{S:4}, we study finite type IFSs of contractive similitudes and prove Theorem \ref{thm(0.5)}. Section \ref{S:41} is devoted to the proof of Theorem \ref{thm(41.1)}. Finally, we present some examples of CIFSs and GIFSs satisfying (FTC) and (GFTC) on Riemannian manifolds, respectively.

\section{The weak separation condition}\label{S:1}

Let $M$ be a complete $n$-dimensional smooth Riemannian manifold, $U\subset M$ be open and connected, and let $W\subset U$ be a compact set with $\overline{W^{\circ}}=W$.
Recall that a map $S:U\longrightarrow U$ is called \textit{conformal} if $S'(x)$ is a similarity matrix for any $x\in U$. We say that $\{S_{i}\}_{i=1}^{N}$ is a \textit{conformal iterated function system} (CIFS)
on $U$, if
\begin{itemize}
\item[$(a)$] for any $i\in\Sigma$, $S_{i}:U\longrightarrow S_{i}(U)\subset U$ is a conformal $C^{1+\varepsilon}$ diffeomorphism for some $\varepsilon\in(0,1)$;
\item[$(b)$] $S_{i}(W)\subset W$ for any $i\in\Sigma$;
\item[$(c)$] $0<|\det S'_{i}(x)|<1$ for any $i\in\Sigma$ and $x\in U$.
\end{itemize}

Since $M$ is a manifold, we can find an open and connected set $U_{1}$ such that $\overline{U_{1}}$ is compact and $W\subset U_{1}\subset \overline{U_{1}}\subset U$. According to \cite{Patzschke_1997}, conditions $(a)$--$(c)$ together imply the \textit{bounded distortion property} (BDP), without assuming any separation condition, i.e., there exists a constant $C_{1}\geq1$ such that for any $\mathbf{u}\in\Sigma^{*}$ and $x,y\in U_{1}$,
\begin{equation}\label{eq(1.1)}
C_{1}^{-n}\leq\frac{|\det S'_{\mathbf{u}}(x)|}{|\det S'_{\mathbf{u}}(y)|}\leq C_{1}^{n}.
\end{equation}
It follows that for any $\mathbf{u}\in\Sigma^*$,
\begin{equation}\label{eq:r_R_bound}
C_{1}^{-1}\leq\frac{r_{\mathbf u}}{R_{\mathbf u}}\leq\frac{R_{\mathbf u}}{r_{\mathbf u}}\leq C_{1}.
\end{equation}
Moreover, there exists a constant $C_{2}\geq1$ such that for any $\mathbf{u}\in\Sigma^{\ast}$ and $x,y\in W$,
\begin{equation}\label{eq(1.2)}
C_{2}^{-1}R_{\mathbf{u}}d(x,y)\leq d(S_{\mathbf{u}}(x),S_{\mathbf{u}}(y))\leq C_{2}R_{\mathbf{u}}d(x,y).
\end{equation}
Let $\nu$ be the Riemannian volume measure, and let $A\subset W$ be a measurable set. Denote the Jacobian determinant of a function $f$ by $\mathbf{J}f$. Then
\begin{equation}\label{eq(1.02)}
\nu\big(S_{\mathbf{u}}(A)\big)=\int_{A}\big|\mathbf{J}\big(S_{\mathbf{u}}(x)\big)\big|d\nu=\int_{A}\big|\det S'_{\mathbf{u}}(x)\big|d\nu.
\end{equation}
(see, e.g. \cite[Proposition 8.1.10]{Abraham-Marsden-Ratiu_2007}).

Note that for any $x\in W$ and $\mathbf{u},\mathbf{v}\in\Sigma^{\ast}$,
$$|\det S'_{\mathbf{u}\mathbf{v}}(x)|=|\det S'_{\mathbf{u}}(S_{\mathbf{v}}(x))\cdot S'_{\mathbf{v}}(x)|=|\det S'_{\mathbf{u}}(S_{\mathbf{v}}(x))|\cdot|\det S'_{\mathbf{v}}(x)|.$$
Hence $R_{\mathbf{u}\mathbf{v}}\leq R_{\mathbf{u}}R_{\mathbf{v}}$ and $r_{\mathbf{u}\mathbf{v}}\geq r_{\mathbf{u}}r_{\mathbf{v}}$. In particular, for $S=S_{\mathbf{u}}\in\mathcal{A}_{b},\mathbf{u}=(u_{1},\dots,u_{k})\in\mathcal{W}_{b}$,
$$\begin{aligned}
br&<R_{u_{1},\dots,u_{k-1}}r_{u_{k}}\\
&\leq C_{1}r_{u_{1},\dots,u_{k-1}}r_{u_{k}}\quad\text{(by (\ref{eq:r_R_bound}))}\\
&\leq C_{1}r_{\mathbf{u}}\\
&\leq C_{1}R_{\mathbf{u}},
\end{aligned}$$
i.e.,
\begin{equation}\label{eq(1.9)}
b<\frac{C_{1}}{r}R_{S}\quad\text{for all }S\in\mathcal{A}_{b}.
\end{equation}

\begin{lem}\label{lem(1.1)}
Let $M$ be a complete $n$-dimensional smooth orientable Riemannian manifold, $U\subset M$ be open and connected, and let $W\subset U$ be a compact set with $\overline{W^{\circ}}=W$. Let $\{S_{i}\}_{i=1}^{N}$ be a CIFS on $U$ defined as above. Then the following hold.
\begin{itemize}
\item[$(a)$] For any $\mathbf{u}\in \mathcal{W}_{b},~0<b\leq1$, and any measurable set $A\subset W$,
$$\bigg(\frac{br}{C_{1}}\bigg)^{n}\nu(A)\leq\nu(S_{\mathbf{u}}(A))\leq b^{n}\nu(A).$$
\item[$(b)$] For any $\mathbf{u},\mathbf{v}\in \mathcal{W}_{b},~0<b\leq1$, and any measurable set $A\subset W$,
$$\bigg(\frac{r}{C_{1}}\bigg)^{n}\nu(S_{\mathbf{v}}(A))\leq\nu(S_{\mathbf{u}}(A))\leq \bigg(\frac{C_{1}}{r}\bigg)^{n}\nu(S_{\mathbf{v}}(A)).$$
\end{itemize}
\end{lem}
\begin{proof}
$(a)$ Let $\mathbf{u}=(u_{1},\dots,u_{k})\in \mathcal{W}_{b}$. Then by the definition of $\mathcal{W}_{b}$,
\begin{equation}\label{eq(1.3)}
\sup_{x\in W}|\det S'_{\mathbf{u}}(x)|\leq b^{n}\leq\sup_{x\in W}|\det S'_{\mathbf{u}^{-}}(x)|.
\end{equation}
Hence for any $x\in W$,
$$\begin{aligned}
b^{n}&\geq |\det S'_{\mathbf{u}}(x)|=|\det S'_{\mathbf{u}^{-}}(S_{u_{k}}(x))|\cdot|\det S'_{u_{k}}(x)|\quad\text{(by (\ref{eq(1.3)}))}\\
&\geq r^{n}\inf_{x\in W}|\det S'_{\mathbf{u}^{-}}(S_{u_{k}}(x))|\\
&\geq \bigg(\frac{r}{C_{1}}\bigg)^{n}\sup_{x\in W}|\det S'_{\mathbf{u}^{-}}(S_{u_{k}}(x))|\quad\text{(by (\ref{eq(1.1)}))}\\
&>\bigg(\frac{rb}{C_{1}}\bigg)^{n}\quad\text{(by (\ref{eq(1.3)}))}.
\end{aligned}$$
It follows that
$$\int_{A}b^{n}d\nu\geq\int_{A}|\det S'_{\mathbf{u}}(x)|d\nu\geq\int_{A}\bigg(\frac{rb}{C_{1}}\bigg)^{n}d\nu,$$
which proves $(a)$ by (\ref{eq(1.02)}).

$(b)$ Making use of part $(a)$, we have
$$\nu(S_{\mathbf{u}}(A))\leq b^{n}\nu(A)\leq\bigg(\frac{C_{1}}{r}\bigg)^{n}\nu(S_{\mathbf{v}}(A)).$$
Similarly,
$$\nu(S_{\mathbf{v}}(A))\leq b^{n}\nu(A)\leq\bigg(\frac{C_{1}}{r}\bigg)^{n}\nu(S_{\mathbf{u}}(A)),$$
which proves $(b)$.
\end{proof}

Let $\mathcal{F}:=\{S_{\mathbf{v}}S_{\mathbf{u}}^{-1}:\mathbf{u},\mathbf{v}\in\Sigma^{\ast}\}$. It is possible that $\tau=S_{\mathbf{v}}S_{\mathbf{u}}^{-1}$ can be simplified to $S_{\mathbf{v}'}S_{\mathbf{u}'}^{-1}$. Thus the domain of
$\tau$ is $S_{\mathbf{u}'}(W)$ (containing $S_{\mathbf{u}}(W)$). Denote the domain of $\tau$ by ${\rm Dom}(\tau)$. The proof of the following lemma is similar to that of \cite[Lemma 2.2]{Lau-Ngai-Wang_2009} and is omitted.

\begin{lem}\label{lem(1.2)}
Assume the same hypotheses of Lemma \ref{lem(1.1)}. Then for any $\mathbf{u},\mathbf{v}\in\Sigma^{\ast}$ and any $x,y\in{\rm Dom}(S_{\mathbf{v}'}S_{\mathbf{u}'}^{-1})$, we have
$$\frac{|\det (S_{\mathbf{v}}S_{\mathbf{u}}^{-1})'(x)|}{|\det (S_{\mathbf{v}}S_{\mathbf{u}}^{-1})'(y)|}\leq C_{1}^{2n}.$$
\end{lem}

\begin{lem}\label{lem(1.3)}
Assume the same hypotheses of Lemma \ref{lem(1.1)}. Let $\tau=S_{\mathbf{v}}S_{\mathbf{u}}^{-1}=S_{\mathbf{v}'}S_{\mathbf{u}'}^{-1}\in\mathcal{F}$ with ${\rm Dom}(\tau)=S_{\mathbf{u}'}(W)$. Then the following hold.
\begin{itemize}
\item[$(a)$] For any measurable subset $A\subset {\rm Dom}(\tau)$,
$$\bigg(\frac{r_{\mathbf{v}'}}{R_{\mathbf{u}'}}\bigg)^{n}\nu(A)\leq \nu(\tau(A))\leq\bigg(\frac{R_{\mathbf{v}'}}{r_{\mathbf{u}'}}\bigg)^{n}\nu(A).$$
\item[$(b)$] For any $A,B$ belonging to some collection $\mathcal{C}$ of measurable subsets of $W$, suppose $C\geq1$ is a constant such that
$$C^{-1}\nu(B)\leq\nu(A)\leq C\nu(B).$$
Then for any $A,B\in\mathcal{C}$ such that $A,B\subset{\rm Dom}(\tau)$,
$$C^{-1}C_{1}^{-2n}\nu(\tau(B))\leq\nu(\tau(A))\leq CC_{1}^{2n}\nu(\tau(B)).$$
\end{itemize}
\end{lem}
\begin{proof}
$(a)$ For $x\in {\rm Dom}(\tau)$, let $y=S_{\mathbf{u}'}^{-1}(x)\in W$. Then
$$|\det\tau'(x)|=|\det (S_{\mathbf{v}'}S_{\mathbf{u}'}^{-1})'(x)|=|\det S'_{\mathbf{v}'}(S_{\mathbf{u}'}^{-1}(x))|\cdot|\det (S_{\mathbf{u}'}^{-1})'(x)|=\frac{|\det S'_{\mathbf{v}'}(y)|}{|\det S'_{\mathbf{u}'}(y)|}.$$
Hence
$$\bigg(\frac{r_{\mathbf{v}'}}{R_{\mathbf{u}'}}\bigg)^{n}\leq
|\det\tau'(x)|\leq\bigg(\frac{R_{\mathbf{v}'}}{r_{\mathbf{u}'}}\bigg)^{n}.$$
Thus,
$$\int_{A}\bigg(\frac{r_{\mathbf{v}'}}{R_{\mathbf{u}'}}\bigg)^{n}d\nu\leq
\int_{A}|\det\tau'(x)|d\nu\leq\int_{A}\bigg(\frac{R_{\mathbf{v}'}}{r_{\mathbf{u}'}}\bigg)^{n}d\nu,$$
which proves $(a)$ by (\ref{eq(1.02)}).

$(b)$ Making use of part $(a)$ and (\ref{eq(1.1)}), we have
$$\begin{aligned}
\nu(\tau(A))&\leq\bigg(\frac{R_{\mathbf{v}'}}{r_{\mathbf{u}'}}\bigg)^{n}\nu(A)
\leq C\bigg(\frac{R_{\mathbf{v}'}}{r_{\mathbf{u}'}}\bigg)^{n}\nu(B)\\
&\leq C\bigg(\frac{R_{\mathbf{v}'}}{r_{\mathbf{u}'}}\bigg)^{2n}
\nu(\tau(B))\quad\text{(by part $(a)$)}\\
&\leq CC_{1}^{2n}\nu(\tau(B))\quad\text{(by (\ref{eq:r_R_bound}))}.
\end{aligned}$$
On the other hand,
$$\begin{aligned}
\nu(\tau(A))&\geq\bigg(\frac{r_{\mathbf{v}'}}{R_{\mathbf{u}'}}\bigg)^{n}\nu(A)
\geq C^{-1}\bigg(\frac{r_{\mathbf{v}'}}{R_{\mathbf{u}'}}\bigg)^{n}\nu(B)\\
&\geq C^{-1}\bigg(\frac{r_{\mathbf{v}'}}{R_{\mathbf{u}'}}\bigg)^{2n}
\nu(\tau(B))\quad\text{(by part $(a)$)}\\
&\geq C^{-1}C_{1}^{-2n}\nu(\tau(B))\quad\text{(by (\ref{eq:r_R_bound}))}.
\end{aligned}$$
This proves $(b)$.
\end{proof}

\begin{lem}\label{lem(1.30)} (Bishop-Gromov inequality (see, e.g. \cite{Anderson_1990,Berger_2003,Bishop-Crittenden}))\label{lem(1.00)}
Let $M$ be a complete $n$-dimensional Riemannian manifold with non-negative Ricci curvature, and $B_{r}(x)$ be an $r$-ball in $M$. Then
$$\nu(B_{r}(x))\leq c_{n}r^{n},$$
where $c_{n}=\pi^{\frac{n}{2}}/\Gamma(\frac{n}{2}+1)$ is the volume of the unit ball in $\mathbb{R}^{n}$.
\end{lem}

The following proposition generalizes \cite[Proposition 3.1]{Lau-Ngai-Wang_2009} to manifolds.

\begin{prop}\label{prop(1.1)}
Let $M$ be a complete $n$-dimensional orientable Riemannian manifold with non-negative Ricci curvature. Assume the same hypotheses of Lemma \ref{lem(1.1)}. Then the following are equivalent:
\begin{itemize}
\item[$(a)$] $\{S_{i}\}_{i=1}^{N}$ satisfies (WSC);
\item[$(b)$] there exists $a>0$ and a nonempty subset $D\subset W$ such that $\gamma_{a,D}<\infty$;
\item[$(c)$] for any $a>0$ and any nonempty subset $D\subset W$, $\gamma_{a,D}<\infty$;
\item[$(d)$] for any $D\subset W$ there exists $\gamma=\gamma(D)$ (depending only on $D$) such that for any $0<b\leq1$ and $x\in W$,
$$\#\{S\in\mathcal{A}_{b}:x\in S(D)\}\leq\gamma.$$
\end{itemize}
\end{prop}
\begin{proof}
$(a)\Rightarrow(b)$ It suffices to prove that there exists $\gamma'\in\mathbb{N}$ and $D\subset W$ with $D^{\circ}\neq\emptyset$, such that for any $x\in W$ and $0<b\leq1$,
$$\#\{S\in\mathcal{A}_{b}:S(D)\cap B_{b}(x)\neq\emptyset\}\leq\gamma'.$$
To obtain this equality, let $D$ be as in the definition of WSC, $S\in\mathcal{A}_{b}$ such that $S(D)\cap B_{b}(x)\neq\emptyset$, and $x'\in D$ such that
$S(x')\in B_{b}(x)$. Then for any $y\in D$,
$$\begin{aligned}
d(S(y),x)&\leq d(S(y),S(x'))+d(S(x'),x)\quad\text{(triangle inequality)}\\
&\leq C_{2}R_{S}d(y,x')+b\quad\text{(by (\ref{eq(1.2)}))}\\
&\leq (1+C_{2}|D|)b.
\end{aligned}$$
Let $\eta:=(1+C_{2}|D|)b$. We can rewrite the above as $S(D)\subset B_{\eta}(x)$. By $(a)$, each point in $W$ is covered by at most $\gamma$ of the $S(D)$, where
$S\in\mathcal{A}_{b}$. It follows that
\begin{equation}\label{eq(1.4)}
\sum\{\nu(S(D)):S\in\mathcal{A}_{b},S(D)\cap B_{b}(x)\neq\emptyset\}\leq\gamma\nu(B_{\eta}(x)).
\end{equation}
Making use of Lemma \ref{lem(1.1)}$(a)$, we have
$$\begin{aligned}
\bigg(\frac{br}{C_{1}}\bigg)^{n}\nu(D)\#\{S\in\mathcal{A}_{b}:S(D)\cap B_{b}(x)\neq\emptyset\}
&\leq\sum\{\nu(S(D)):S\in\mathcal{A}_{b},S(D)\cap B_{b}(x)\neq\emptyset\}\\
&\leq\gamma\nu(B_{\eta}(x))\quad\text{(by (\ref{eq(1.4)}))}\\
&\leq\gamma c_{n}\eta^{n}\quad\text{(by Lemma \ref{lem(1.00)})}\\
&:=\gamma cb^{n},
\end{aligned}$$
where $c:=c_{n}(1+C_{2}|D|)^{n}$.
Hence there exists $\gamma'\in\mathbb{N}$ such that
$$\#\{S\in\mathcal{A}_{b}:S(D)\cap B_{b}(x)\neq\emptyset\}\leq\frac{\gamma c C_{1}^{n}}{r^{n}\nu(D)}\leq\gamma'.$$

$(b)\Rightarrow(c)$ We prove the contrapositive. Assume $(c)$ is false. Then there exist $a_{0}>0$ and a nonempty subset $D_{0}\subset W$ such that $\gamma_{a_{0},D_{0}}=\infty$. Hence there exists a sequence $\{A_{k}\}_{k=1}^{\infty}$ of nonempty sets bounded in $M$ such that
\begin{equation}\label{eq(1.5)}
\{S\in\mathcal{A}_{a_{0}|A_{k}|}:S(D_{0})\cap A_{k}\neq\emptyset\}\geq k.
\end{equation}
To prove $(b)$ fails, we fix an arbitrary $a>0$ and nonempty subset $D\subset W$. We will show $\gamma_{a,D}=\infty$. Let
$$s:=\sup_{x\in D_{0},y\in D}d(x,y)<\infty.$$
We first claim that for any $S\in\mathcal{A}_{a_{0}|A_{k}|}$ and $\delta_{k}:=s a_{0} C_{2}|A_{k}|$, $S(D_{0})\cap A_{k}\neq\emptyset$ implies $S(D)\cap (A_{k})_{\delta_{k}}\neq\emptyset$, where $(A_{k})_{\delta_{k}}=\{x\in M:{\rm dist}(x,A_{k})\leq\delta_{k}\}$ is the closed $\delta_{k}$-neighborhood of $A_{k}$.
To prove the claim, we let $y\in S(D_{0})\cap A_{k}$. Then there exists $x\in D_{0}$ such that $y=S(x)\in S(D_{0})$. Now let $\tilde{x}\in D$ and $\tilde{y}:=S(\tilde{x})\in S(D)$. It follows from (\ref{eq(1.2)}) that
$$d(\tilde{y},y)=d(S(\tilde{x}),S(x))\leq C_{2}R_{S}d(\tilde{x},x)\leq s C_{2}R_{S}\leq s a_{0}|C_{2}A_{k}|=\delta_{k}.$$
This proves the claim. Note that $(A_{k})_{\delta_{k}}$ is a set of diameter $2\delta_{k}+|A_{k}|=(2sa_{0} C_{2}+1)|A_{k}|$. Since $M$ is a Riemannian manifold with non-negative Ricci curvature, $M$ has the doubling property. Hence we can cover $(A_{k})_{\delta_{k}}$ by no more than $\lambda$ sets of diameter $(a_{0}|A_{k}|)/a$. Note that
$\mathcal{A}_{a_{0}|A_{k}|}=\mathcal{A}_{(a_{0}|A_{k}|)/a}$. It follows from (\ref{eq(1.5)}) and the claim in the above that there exists $A_{k}^{\ast}\subset M$ with
$|A_{k}^{\ast}|=(a_{0}|A_{k}|)/a$ such that
$$\{S\in\mathcal{A}_{a_{0}|A_{k}^{\ast}|}:S(D_{0})\cap A_{k}^{\ast}\neq\emptyset\}\geq \frac{k}{\lambda}.$$
Since $\lambda$ is independent of $k$, we conclude that $\gamma_{a,D}=\infty$.

$(c)\Rightarrow(d)$ Let $D\subset W$. Then for any $x\in W$ and $0<b\leq1$,
$$\begin{aligned}
\#\{S\in\mathcal{A}_{b}:x\in S(D)\}&\leq\#\{S\in\mathcal{A}_{b},S(D)\cap B_{b/2}(x)\neq\emptyset\}\\
&=\#\mathcal{A}_{1,B_{b/2}(x),D}\\
&\leq\gamma_{1,D}<\infty.
\end{aligned}$$

$(d)\Rightarrow(a)$ is trivial.
\end{proof}

\begin{proof}[Proof of Theorem \ref{thm(0.1)}]  Use of the above lemmas and propositions, as in \cite[Theorem 3.2]{Lau-Ngai-Wang_2009}; we omit the details.
\end{proof}

As an important consequence of Theorem \ref{thm(0.1)}, we obtain the following estimate for $\#\mathcal{A}_{b}$ and a formula for $\dim_{{\rm H}}(F)$. Making use of (\ref{eq(1.2)}), (\ref{eq(1.9)}) and Proposition \ref{prop(1.1)}, we obtain the following analogue of \cite[Corollary 3.3]{Lau-Ngai-Wang_2009}. The proof is similar and is omitted.

\begin{coro}\label{coro(1.1)}
Let $\{S_{i}\}_{i=1}^{N}$ and $\alpha$ be as in Theorem \ref{thm(0.1)}. Then there exists a constant $C_{3}>0$ such that for any $0<b\leq1$,
\begin{equation}\label{eq(1.10)}
C_{3}^{-1}b^{-\alpha}\leq\#\mathcal{A}_{b}\leq C_{3}b^{\alpha}.
\end{equation}
Consequently,
$$\alpha=\dim_{{\rm H}}(K)=\dim_{{\rm B}}(K)=\dim_{{\rm P}}(K)=-\lim_{b\rightarrow0^{+}}\frac{\log\#\mathcal{A}_{b}}{\log b}.$$
\end{coro}

\section{Absolute continuity of self-conformal measures}\label{S:2}
Throughout this section, we let $M$ be a complete $n$-dimensional smooth orientable Riemannian manifold with non-negative Ricci curvature, $U\subset M$ be open and connected, and $W\subset U$ be a compact set with $\overline{W^{\circ}}=W$. The proof of the following proposition is similar to that of \cite[Proposition 4.1]{Lau-Ngai-Wang_2009} and is omitted.

\begin{prop}\label{prop(2.1)}
Assume that $\{S_{i}\}_{i=1}^{N}$ is a CIFS on $U$ satisfying (WSC). Then for any finite subset $\Lambda\subset\Sigma^{\ast}$, the family $\{S_{\mathbf{v}}:\mathbf{v}\in\Lambda\}$ also satisfies (WSC).
\end{prop}

For $\mathbf{u}\in\Sigma^{\ast}$, let $[\mathbf{u}]:=\{\mathbf{u}'\in\Sigma^{\ast}:S_{\mathbf{u}}=S_{\mathbf{u}'}\}$. The following lemma can be proved by using Corollary \ref{coro(1.1)} and the argument in \cite[Lemma 4.2]{Lau-Ngai-Wang_2009}.

\begin{lem}\label{lem(2.1)}
Assume that $\{S_{i}\}_{i=1}^{N}$ is a CIFS on $U$ satisfying (WSC). Let $\{p_{i}\}_{i=1}^{N}$ be the associated probability weights, and let $K$ be the attractor
with $\dim_{{\rm H}}(K)=\alpha$. For $0<b\leq1$ and $\Lambda\subset\mathcal{W}_{b}$, let
$$\widetilde{\Lambda}:=\bigg\{\mathbf{u}\in\Lambda:\sum_{\mathbf{u}'\in[\mathbf{u}]
\cap\Lambda}p_{\mathbf{u}'}>\frac{b^{\alpha}}{4C_{3}}\bigg\},$$
where $C_{3}$ is as in Corollary \ref{coro(1.1)}. Then $P(\Lambda)>1/2$ implies $P(\widetilde{\Lambda})>1/4$.
\end{lem}

\begin{proof}[Proof of Theorem~\ref{thm(0.2)}] The proof of Theorem \ref{thm(0.2)} follows by using Proposition \ref{prop(1.1)}, Lemma \ref{lem(2.1)}, and a technique in the proofs of \cite[Theorem 1.1]{Lau-Ngai-Rao_2001} and \cite[Theorem 1.1]{Lau-Wang_2004}; we omit the details.
\end{proof}

We first give the definition of Vitali relation (see \cite[Definition 2.8.16]{Federer_1969}). Let $X$ be a metric space.
Any subset of
$$
\{(x,A): x\in A\subset X\}
$$
is called a \textit{covering relation}. Let $\mathbf{C}$ be a covering relation and $Z\subset  X$, we let
$$
\mathbf{C}(Z):=\{A\subset X:(x,A)\in \mathbf{C}\text{ for some }x\in Z\}.
$$
We say $\mathbf{C}$ is \textit{fine at $x$} if
$$
\inf\{|A|:(x,A)\in \mathbf{C}\}=0.
$$
We say $\mathbf{C}(Z)$ is \textit{fine on $Z$} if for any $x\in Z$, $\mathbf{C}(x)$ is fine at $x$. Let $\phi$ be a regular Borel measure on $X$. A covering relation $\mathbf{V}$ is called a
$\phi$ \textit{Vitali relation} if $\mathbf{V}(X)$ is a family of Borel sets, $\mathbf{V}$ is fine on $X$, and moreover, whenever $\mathbf{C}\subset \mathbf{V}, Z\subset X$, and $\mathbf{C}$ is fine on $Z$, then $\mathbf{C}(Z)$ has a countable disjoint subfamily covering $\phi$ almost all of $Z$.

For an $n$-dimension Riemannian manifold  $M$, by making use of \cite[Definition 2.8.9 and Theorem 2.8.18]{Federer_1969}, we see that $\{B_{r}(x):x\in M,0<r<\infty\}$ is a $\phi$ Vitali relation in $M$. Let $f$ be the Radon-Nikodym derivative of $\mu$ with respect to $\phi$. By \cite[Theorem 2.9.8]{Federer_1969} or \cite[ Lemma 2.1(i)]{Bedferd-Fisher_1992}, we have
$$\lim_{r\rightarrow0}\frac{1}{\phi(B_{r}(x))}\int_{B_{r}(x)}f d\phi=\lim_{r\rightarrow0}\frac{\mu(B_{r}(x))}{\phi(B_{r}(x))}=f(x)$$
for $\phi$ almost all $x\in K$. Assume that $A\subset K$, and $f=\chi_{A}$. It follows that
\begin{equation}\label{eq(2.4)}
\lim_{r\rightarrow0}\frac{\phi(A\cap B_{r}(x))}{\phi(B_{r}(x))}
=\lim_{r\rightarrow0}\frac{1}{\phi(B_{r}(x))}\int_{B_{r}(x)}\chi_{A} d\phi=1
\end{equation}
for $\phi$ almost all $x\in A$.

\begin{proof}[Proof of Theorem \ref{thm(0.3)}]
Let $\phi:=\mathcal{H}^{\alpha}|_{K}$ and $f$ be the Radon-Nikodym derivative of $\mu$ with respect to $\phi$. Suppose that $f$ is unbounded. For any $Q>0$, let $A=A(Q):=\{t\in K:f(t)>Q\}$. Then for any $\delta>0$, by (\ref{eq(2.4)}), there exist $x\in K$ and $b>0$ such that
\begin{equation}\label{eq(2.5)}
\phi\big(A\cap B_{b\delta}(x)\big)>\frac{1}{2}\phi(B_{b\delta}(x)).
\end{equation}
Then
\begin{equation}\label{eq(2.6)}\begin{aligned}
\mu(B_{b\delta}(x))&=\int_{B_{b\delta}(x)}f(t)d\phi(t)\\
&\geq \int_{A\cap B_{b\delta}(x)}f(t)d\phi(t)\\
&\geq Q\phi(A\cap B_{b\delta}(x))\\
&>\frac{1}{2}Q\phi(B_{b\delta}(x))\quad\text{(by (\ref{eq(2.5)}))}.
\end{aligned}\end{equation}
Let $\delta:=C_{2}|K|$ in the above inequality. Note that $x\in K=\bigcup_{S\in\mathcal{A}_{b}}S(K)$, and hence there exists $S\in\mathcal{A}_{b}$ such that $x\in S(K)$.
Making use of (\ref{eq(1.2)}), we have
$$|S(K)|\leq C_{2}R_{S}|K|\leq b\delta,$$
which implies that $S(K)\subset B_{b\delta}(x)$. For $S=S_{u_{1}\cdots u_{k}}\in\mathcal{A}_{b}$, we have $R_{S}=R_{\mathbf{u}}\leq b<R_{\mathbf{u}^{-}}$. Hence by (\ref{eq(1.1)}){\color{blue},}
\begin{equation}\label{eq(2.7)}
br<R_{u_{1}\cdots u_{k-1}}r_{u_{k}}\leq C_{1}r_{u_{1}\cdots u_{k-1}}r_{u_{k}}\leq C_{1}r_{\mathbf{u}}\leq C_{1}R_{\mathbf{u}}.
\end{equation}
Combining these derivations with (\ref{eq(1.2)}) and (\ref{eq(2.7)}), we have
\begin{equation}\label{eq(2.8)}
\phi(B_{b\delta}(x))\ge\phi(S(K))=\mathcal{H}^{\alpha}(S(K))\geq C_{2}^{-\alpha}R_{S}^{\alpha}\mathcal{H}^{\alpha}(K)>\bigg(\frac{r}{C_{1}C_{2}}\bigg)^{\alpha}b^{\alpha}\phi(K).
\end{equation}
Combining \eqref{eq(2.6)} and \eqref{eq(2.8)}, we obtain
\begin{equation}\label{eq(2.9)}
\mu(B_{b\delta}(x))>\frac{1}{2}Q\bigg(\frac{r}{C_{1}C_{2}}\bigg)^{\alpha}b^{\alpha}\phi(K).
\end{equation}
On the other hand,
$$\begin{aligned}
\mu(B_{b\delta}(x))&=\sum_{S\in\mathcal{A}_{b},S(K)\cap B_{b\delta}(x)\neq\emptyset}p_{S}\mu\circ S^{-1}(B_{b\delta}(x))\\
&\leq\sum_{S\in\mathcal{A}_{b},S(K)\cap B_{b\delta}(x)\neq\emptyset}p_{S}.
\end{aligned}$$
Let $S\in\mathcal{A}_{b}$ such that $p_S$ is the maximum among all the summands in the above summation. Then
\begin{equation}\label{eq(2.10)}
\mu(B_{b\delta}(x))\leq p_{S}\#\{S\in\mathcal{A}_{b}:S(K)\cap B_{b\delta}(x)\neq\emptyset\}\leq p_{S}\gamma_{\frac{1}{2\delta},K}.
\end{equation}
We choose $Q$ such that
\begin{equation}\label{eq(2.11)}
\frac{1}{2}Q\bigg(\frac{r}{C_{1}C_{2}}\bigg)^{\alpha}\phi(K)>\gamma_{\frac{1}{2\delta},K}.
\end{equation}
It follows from \eqref{eq(2.9)}--\eqref{eq(2.11)} that there exists $S\in\mathcal{A}_{b}$ such that
$$b^{\alpha}\gamma_{\frac{1}{2\delta},K}<\mu(B_{b\delta}(x))\leq p_{S}\gamma_{\frac{1}{2\delta},K}.$$
Hence
$$R_{S}^{\alpha}\leq b^{\alpha}<p_{S}.$$
It now follows from Theorem \ref{thm(0.2)} that $\mu$ is singular with respect to $\phi$, a contradiction. This proves the theorem.
\end{proof}

\section{The finite type condition}\label{S:3}

In this section we extend (FTC) to Riemannian manifolds and prove Theorem~\ref{thm(0.4)}. We follow the set-up in \cite{Lau-Ngai_2007, Lau-Ngai-Wang_2009}. We begin by defining a \textit{partial order} $\preceq$ on $\Sigma^{\ast}$. For $\mathbf{u},\mathbf{v}\in\Sigma^{\ast}$, we denote by $\mathbf{u}\preceq\mathbf{v}$ if $\mathbf{u}$ is an initial segment of $\mathbf{v}$ or $\mathbf{u}=\mathbf{v}$. We denote by $\mathbf{u}\npreceq\mathbf{v}$ if $\mathbf{u}\preceq\mathbf{v}$ does not hold. Let
$\{\mathcal{M}_{k}\}_{k=0}^{\infty}$ be a sequence of index sets such that for any $k\geq0$, $\mathcal{M}_{k}$ is a finite subset of $\Sigma^{\ast}$.
We say that $\mathcal{M}_{k}$ is an \textit{antichain} if for any $\mathbf{u},\mathbf{v}\in\mathcal{M}_{k}$, $\mathbf{u}\npreceq\mathbf{v}$ and $\mathbf{v}\npreceq\mathbf{u}$. Let
$$\underline{m}_{k}=\underline{m}_{k}(\mathcal{M}_{k}):=\min\{|\mathbf{u}|:\mathbf{u}\in\mathcal{M}_{k}\},
\quad \overline{m}_{k}=\overline{m}_{k}(\mathcal{M}_{k}):=\max\{|\mathbf{u}|:\mathbf{u}\in\mathcal{M}_{k}\},$$
where $|\mathbf{u}|$ is the length of $\mathbf{u}$.

\begin{defi}\label{defi(3.1)}
We say that $\{\mathcal{M}_{k}\}_{k=0}^{\infty}$ is a \textit{sequence of nested index sets} if it satisfies the following conditions:
\begin{itemize}
\item[$(a)$] both $\{\underline{m}_{k}\}$ and $\{\overline{m}_{k}\}$ are nondecreasing, and
$\displaystyle{\lim_{k\rightarrow\infty}\underline{m}_{k}=\lim_{k\rightarrow\infty}\overline{m}_{k}}=\infty$;
\item[$(b)$] for any $k\geq0$, $\mathcal{M}_{k}$ is an antichain in $\Sigma^{\ast}$;
\item[$(c)$] for any $\mathbf{v}\in\Sigma^{\ast}$ with $|\mathbf{v}|>\overline{m}_{k}$, there exists $\mathbf{u}\in\mathcal{M}_{k}$ such that $\mathbf{u}\preceq\mathbf{v}$;
\item[$(d)$] for any $\mathbf{v}\in\Sigma^{\ast}$ with $|\mathbf{v}|<\underline{m}_{k}$, there exists $\mathbf{u}\in\mathcal{M}_{k}$ such that $\mathbf{v}\preceq\mathbf{u}$;
\item[$(e)$] there exists a positive integer $L$, independent of $k$, such that for any $\mathbf{u}\in\mathcal{M}_{k}$ and $\mathbf{v}\in\mathcal{M}_{k+1}$ with
$\mathbf{u}\preceq\mathbf{v}$, we have $|\mathbf{v}|-|\mathbf{u}|\leq L$.
\end{itemize}
\end{defi}

Note that $\mathcal{M}_{k}$ can intersect $\mathcal{M}_{k+1}$, and $\{\Sigma^{k}\}_{k=0}^{\infty}$ is an example of a sequence of nested index sets.
For each integer $k\geq0$, let $\mathcal{V}_{k}$ be the set of \textit{$k$-th level vertices} defined as
$$\mathcal{V}_{0}:=\{(\mathbf{u},0)\}\quad\text{and}\quad\mathcal{V}_{k}
:=\{(S_{\mathbf{u}},k):\mathbf{u}\in\mathcal{M}_{k}\}\text{ for any }k\geq1.$$
We write $\omega_{{\rm root}}:=(\mathbf{u},0)$ and call it the \textit{root vertex}. Let $\mathcal{V}:=\bigcup_{k\geq0}\mathcal{V}_{k}$ be the set of all vertices. For
$\omega=(S_{\mathbf{u}},k)$, we define $S_{\omega}:=S_{\mathbf{u}}$. Let $W\subset M$ be a compact set. For an IFS $\{S_{i}\}_{i=1}^{N}$ on $W$, let $\Omega\subset W$ be a nonempty open set that is invariant under $\{S_{i}\}_{i=1}^{N}$. Such a set exists if $\{S_{i}\}_{i=1}^{N}$ are contractions on $W$. We say that two $k$-th level vertices $\omega,\omega'\in\mathcal{V}_{k}$ are \textit{neighbors} if $S_{\omega}(\Omega)\cap S_{\omega'}(\Omega)\neq\emptyset$. Let
$$\Omega(\omega):=\{\omega':\omega'\in\mathcal{V}_{k}\text{ is a neighbor of }\omega\},$$
which is called the \textit{neighborhood} of $\omega$.

\begin{defi}\label{defi(3.2)}
For any two vertices $\omega\in\mathcal{V}_{k}$ and $\omega'\in\mathcal{V}_{k'}$, let
$$\tau:=S_{\omega'}S_{\omega}^{-1}:\bigcup_{\sigma\in\Omega(\omega)}S_{\sigma}(W)\longrightarrow W.$$
We say $\omega$ and $\omega'$ are \textit{equivalent}, i.e., $\omega\sim\omega'$, if the following conditions hold
\begin{itemize}
\item[$(a)$] $\{S_{\sigma'}:\sigma'\in\Omega(\omega')\}=\{\tau S_{\sigma}:\sigma\in\Omega(\omega)\}$;
\item[$(b)$] for $\sigma\in\Omega(\omega)$ and $\sigma'\in\Omega(\omega')$ such that $S_{\sigma'}=\tau S_{\sigma}$, and for any positive integer $k_{0}\geq1$,
$\mathbf{u}\in\Sigma^{\ast}$, $\mathbf{u}$ satisfies $(S_{\sigma}\circ S_{\mathbf{u}},k+k_{0})\in\mathcal{V}_{k+k_{0}}$ if and only if it satisfies
$(S_{\sigma'}\circ S_{\mathbf{u}},k'+k_{0})\in\mathcal{V}_{k'+k_{0}}$.
\end{itemize}
\end{defi}
It is easy to see that $\sim$ is an equivalence relation. Denote the equivalent class of $\omega$ by $[\omega]$, and call it the \textit{neighborhood types} of $\omega$.
\par Let $\omega=(S_{\mathbf{u}},k)\in\mathcal{V}_{k}$ and $\sigma=(S_{\mathbf{v}},k+1)\in\mathcal{V}_{k+1}$. Suppose that there exists $\mathbf{w}\in\Sigma^{\ast}$ such that
$$\mathbf{v}=(\mathbf{u},\mathbf{w}).$$
Then we connect a \textit{directed edge} from $\omega$ to $\sigma$, and denote this as $\omega\stackrel{\mathbf{w}}{\longrightarrow}\sigma$. We call $\omega$ a \textit{parent} of $\sigma$ and $\sigma$ an \textit{offspring} of $\omega$. Define a graph $\mathcal{G}:=(\mathcal{V},\mathcal{E})$, where $\mathcal{E}$ is the set of all directed edges. We first remove from $\mathcal{G}$ all but the smallest (in the lexicographic order) directed edges going to a vertex. After that, we remove all vertices that do not have any offspring, together with all vertices and edges leading only to them. The resulting graph is called the \textit{reduced graph}. Denote it by $\mathcal{G}_{R}:=(\mathcal{V}_{R},\mathcal{E}_{R})$, where $\mathcal{V}_{R}$ and $\mathcal{E}_{R}$ are the sets of all vertices and all edges, respectively.

The proof of the following proposition is similar to that of \cite[Proposition 2.4]{Lau-Ngai_2007}; we omit the details.

\begin{prop}\label{prop(3.1)}
Let $\omega$ and $\omega'$ be two vertices in $\mathcal{V}$ with offspring $\mathbf{u}_{1},\dots,\mathbf{u}_{m}$ and $\mathbf{u}'_{1},\dots,\mathbf{u}'_{s}$ in $\mathcal{G}_{R}$, respectively. Suppose $[\omega]=[\omega']$. Then
$$\big\{[\mathbf{u}_{i}]:1\leq i\leq m\big\}=\big\{[\mathbf{u}'_{i}]:1\leq i\leq s\big\}$$
counting multiplicity. In particular, $m=s$.
\end{prop}

\begin{defi}\label{defi(3.3)}
Let $\{S_{i}\}_{i=1}^{N}$ be an IFS on $W$ consisting of injective contractions, and let $\mathcal{V}/_{\sim}:=\{[\omega]:\omega\in\mathcal{V}\}$. We say that $\{S_{i}\}_{i=1}^{N}$ satisfies the \textit{finite type condition} (FTC) if there exists a nonempty invariant open set $\Omega\subset W$ with respect to some sequence of nested index sets $\{\mathcal{M}_{k}\}_{k=0}^{\infty}$ and such that
$$\#\mathcal{V}/_{\sim}<\infty.$$
Such a set $\Omega$ is called \textit{a finite type condition set (FTC set)}.
\end{defi}

Obviously, if $\{S_{i}\}_{i=1}^{N}$ satisfies (OSC), then $\#\mathcal{V}/_{\sim}=1$, and thus $\{S_{i}\}_{i=1}^{N}$ satisfies (FTC). We assume that $\{S_{i}\}_{i=1}^{N}$ is a CIFS in the rest of this section.

\begin{lem}\label{lem(3.1)}
Let $\{S_{i}\}_{i=1}^{N}$ be a CIFS on a compact subset $W\subset M$. Assume that $\{S_{i}\}_{i=1}^{N}$ satisfies (FTC) with $\Omega\subset W$ being an FTC set. Then there
exists a constant $C_{4}\geq1$ such that for any two neighboring vertices $\omega_{1}$ and $\omega_{2}$, we have
$$C_{4}^{-1}\leq\frac{\nu(S_{\omega_{1}}(\Omega))}{\nu(S_{\omega_{2}}(\Omega))}\leq C_{4}.$$
\end{lem}
\begin{proof}
Let $\mathcal{T}$ be a neighborhood type, and $\omega$ be a vertex such that $[\omega]=\mathcal{T}$. Let
$$\Omega(\omega)=\{\omega_{0},\omega_{1},\dots,\omega_{m}\},$$
where $\omega_{0}=\omega$. Substituting $S_{\omega_{0}}(\Omega)=A$ and $S_{\omega_{i}}S_{\omega_{0}}^{-1}=\tau$ into Lemma \ref{lem(1.3)}$(a)$ and using \eqref{eq:r_R_bound}, we see that there exists a constant $c_{1}\geq1$ such that for any $i\in\{0,1,\dots,m\}$,
\begin{equation}\label{eq(3.1)}
c_{1}^{-1}\nu(S_{\omega_{0}}(\Omega))\leq\nu(S_{\omega_{i}}(\Omega))\leq c_{1}\nu(S_{\omega_{0}}(\Omega)).
\end{equation}
Let $\omega\sim\omega'$, $\tau=S_{\omega'}S_{\omega}^{-1}\in \mathcal{F}$, and
$$\Omega(\omega')=\{\omega'_{0},\omega'_{1},\dots,\omega'_{m}\},$$
where $\omega'_{0}=\omega'$. Without loss of generality, for any $i\in\{0,1,\dots,m\}$ we can assume $S_{\omega'_{i}}=\tau S_{\omega_{i}}$. It follows from the definition of $\tau$ that
$S_{\omega_{i}}(\Omega)\subset{\rm Dom}(\tau)$. Making use of (\ref{eq(3.1)}) and substituting $S_{\omega_{i}}(\Omega)=A$, $S_{\omega_{0}}(\Omega)=B$ and $S_{\omega'_{i}}S_{\omega_{i}}^{-1}=\tau$ into Lemma \ref{lem(1.3)}$(b)$, we see that for any $i\in\{0,1,\dots,m\}$,
$$c_{1}^{-1}C_{1}^{-2n}\nu(S_{\omega'_{0}}(\Omega))\leq\nu(S_{\omega'_{i}}(\Omega))=\nu(\tau S_{\omega_{i}}(\Omega))\leq c_{1}C_{1}^{2n}\nu(S_{\omega'_{0}}(\Omega)).$$
Hence the lemma holds for any two neighboring vertices $\omega_{1},\omega_{2}$ with one of them being of type $\mathcal{T}$. Since there are only finitely
many distinct neighborhood types, the result follows.
\end{proof}

\begin{lem}\label{lem(3.2)}
Let $\{S_{i}\}_{i=1}^{N}$ be a CIFS on a compact subset $W\subset M$. Then for any $\mathbf{u}\in\Sigma^{k}$ and $\Omega\subset W$, we have
\begin{equation}\label{eq(3.01)}
r^{kn}\leq\frac{\nu(S_{\mathbf{u}}(\Omega))}{\nu(\Omega)}\leq R^{kn}.
\end{equation}
\end{lem}
\begin{proof}
Let $x\in W$. Then by the definition of $R$, we have
$$\begin{aligned}
|\det S'_{\mathbf{u}}(x)|&=|\det S'_{\mathbf{u}^{-}}(S_{u_{k}}(x))|\cdot|S'_{u_{k}}(x)|\\
&\leq R_{\mathbf{u}^{-}}^{n}R_{u_{k}}^{n}\\
&\leq R_{u_{1}}^{n}\cdots R_{u_{k}}^{n}\\
&\leq R^{kn}.
\end{aligned}$$
For any set $\Omega\subset W$, making use of (\ref{eq(1.02)}), we have
$$\nu(S_{\mathbf{u}}(\Omega))=\int_{\Omega}|\det S'_{\mathbf{u}}(x)|d\nu(x)\leq R^{kn}\nu(\Omega).$$
This proves the right side of (\ref{eq(3.01)}). On the other hand, if $x\in W$, then by the definition of $r$, we have
$$\begin{aligned}
|\det S'_{\mathbf{u}}(x)|&=|\det S'_{\mathbf{u}^{-}}(S_{u_{k}}(x))|\cdot|S'_{u_{k}}(x)|\\
&\geq r_{\mathbf{u}^{-}}^{n}r_{u_{k}}^{n}\\
&\geq r_{u_{1}}^{n}\cdots r_{u_{k}}^{n}\\
&\geq r^{kn}.
\end{aligned}$$
Consequently,
$$\nu(S_{\mathbf{u}}(\Omega))=\int_{\Omega}|\det S'_{\mathbf{u}}(x)|d\nu(x)\geq r^{kn}\nu(\Omega).$$
This proves the left side of (\ref{eq(3.01)}).
\end{proof}

We now prove Theorem \ref{thm(0.4)}.

\begin{proof}[Proof of Theorem \ref{thm(0.4)}]
For $0<b\leq1$, $x\in W$, let
$\mathcal{S}(\Omega):=\{S\in\mathcal{A}_{b}:x\in S(\Omega)\}.$
List all elements of $\mathcal{S}(\Omega)$ as $S_{\mathbf{u}_{1}},\dots,S_{\mathbf{u}_{m}}$. For any $j\in\{1,\dots,m\}$, note that $\mathbf{u}_{j}\in\mathcal{M}_{k}$ for some $k$. Let $\mathbf{u}_{j}=(\widetilde{\mathbf{u}}_{j},\widetilde{\mathbf{v}}_{j})$, where
$\widetilde{\mathbf{u}}_{j}\in\mathcal{M}_{k_{j}}$ is the longest initial segment of $\mathbf{u}_{j}$. Without
loss of generality, we assume
$$k_{1}=\min\{k_{j}:1\leq j\leq m\}=k.$$
Then $\widetilde{\mathbf{u}}_{1}\in\mathcal{M}_{k}$. For any $j\in\{1,\dots,m\}$, let $\mathbf{u}'_{j}$ be the initial segment of $\mathbf{u}_{j}$ such that
$\mathbf{u}'_{j}\in\mathcal{M}_{k}$. Note that $\mathbf{u}'_{1}=\widetilde{\mathbf{u}}_{1}$. Since $x\in S(\Omega)$ for any $S\in\mathcal{S}_{\Omega}$, we have
$$\omega_{2}=(S_{\mathbf{u}'_{2}},k),\dots,\omega_{m}=(S_{\mathbf{u}'_{m}},k)$$
are neighbors of $\omega_{1}=(S_{\mathbf{u}'_{1}},k)$. It follows from the definition of (FTC) that there exists a positive integer $c_{2}$ independent of $x,b$ such that
\begin{equation}\label{eq(3.2)}
\#\{\omega_{1},\dots,\omega_{m}\}\leq c_{2}.
\end{equation}
By Lemma \ref{lem(3.1)}, for any $j\in\{2,\dots,m\}$, we have
\begin{equation}\label{eq(3.3)}
C_{4}^{-1}\leq\frac{\nu(S_{\mathbf{u}'_{1}}(\Omega))}{\nu(S_{\mathbf{u}'_{j}}(\Omega))}\leq C_{4}.
\end{equation}
For any $j\in\{2,\dots,m\}$, making use of Lemma \ref{lem(1.1)}$(b)$,  we have
\begin{equation}\label{eq(3.4)}
\bigg(\frac{r}{C_{1}}\bigg)^{n}\leq\frac{\nu(S_{\mathbf{u}_{j}}(\Omega))}{\nu(S_{\mathbf{u}_{1}}(\Omega))}\leq \bigg(\frac{C_{1}}{r}\bigg)^{n}.
\end{equation}
It follows from (\ref{eq(3.3)}) and (\ref{eq(3.4)}) that
\begin{equation}\label{eq(3.5)}
C_{4}^{-1}\bigg(\frac{r}{C_{1}}\bigg)^{n}\leq
\frac{\nu(S_{\mathbf{u}_{j}}(\Omega))}{\nu(S_{\mathbf{u}'_{j}}(\Omega))}
\cdot\frac{\nu(S_{\mathbf{u}'_{1}}(\Omega))}{\nu(S_{\mathbf{u}_{1}}(\Omega))}
\leq C_{4}\bigg(\frac{C_{1}}{r}\bigg)^{n}.
\end{equation}
For each $j\in\{1,\dots,m\}$, write $\mathbf{u}_{j}=\mathbf{u}'_{j}\mathbf{v}_{j}$. Let $S_{\mathbf{I}}$ be the identity map, and $\tau=S_{\mathbf{u}'_{j}}S_{\mathbf{I}}^{-1}$. Then
\begin{equation}\label{eq(3.6)}
\begin{aligned}
\frac{\nu(S_{\mathbf{u}_{j}}(\Omega))}{\nu(S_{\mathbf{u}'_{j}}(\Omega))}
&=\frac{\nu(\tau S_{\mathbf{v}_{j}}(\Omega))}{\nu(\tau(\Omega))}\\
&\leq\frac{R_{\mathbf{u}'_{j}}^{n}\nu(S_{\mathbf{v}_{j}}(\Omega))}{r_{\mathbf{u}'_{j}}^{n}\nu(\Omega)}
\quad\text{(by Lemma \ref{lem(1.3)}$(a)$)}\\
&\leq C_{1}^{n}\frac{\nu(S_{\mathbf{v}_{j}}(\Omega))}{\nu(\Omega)}\quad\text{(by \eqref{eq:r_R_bound})}.
\end{aligned}
\end{equation}
It follows that for any $j\in\{1,\dots,m\}$,
\begin{equation}\label{eq(3.7)}
\begin{aligned}
\frac{\nu(S_{\mathbf{u}_{1}}(\Omega))}{\nu(S_{\mathbf{u}'_{1}}(\Omega))}
&\leq C_{4}\bigg(\frac{C_{1}}{r}\bigg)^{n}\frac{\nu(S_{\mathbf{u}_{j}}(\Omega))}{\nu(S_{\mathbf{u}'_{j}}(\Omega))}
\quad\text{(by (\ref{eq(3.5)}))}\\
&\leq C_{4}\bigg(\frac{C_{1}^{2}}{r}\bigg)^{n}\frac{\nu(S_{\mathbf{v}_{j}}(\Omega))}{\nu(\Omega)}\quad\text{(by (\ref{eq(3.6)}))}\\
&\leq C_{4}\bigg(\frac{C_{1}^{2}}{r}\bigg)^{n}R^{|\mathbf{v}_{j}|n}\quad\text{(by (\ref{eq(3.01)}))}.
\end{aligned}
\end{equation}
On the other hand, similar to (\ref{eq(3.7)}) and by (\ref{eq(3.01)}), we have
\begin{equation}\label{eq(3.8)}
\frac{\nu(S_{\mathbf{u}_{1}}(\Omega))}{\nu(S_{\mathbf{u}'_{1}}(\Omega))}\geq C_{1}^{-n}\frac{\nu(S_{\mathbf{v}_{1}}(\Omega))}{\nu(\Omega)}
\geq C_{1}^{-n}r^{|\mathbf{v}_{j}|n}.
\end{equation}
Recall that $\mathbf{u}'_{1}\in\mathcal{M}_{k}$ is the longest initial segment of $\mathbf{u}_{1}=\mathbf{u}'_{1}\mathbf{v}_{1}$. In view of Definition \ref{defi(3.1)}$(c)$, there exists $\mathbf{v}'_{1}\in\Sigma^{\ast}$ such that $\mathbf{u}'_{1}\mathbf{v}_{1}=\mathbf{u}'_{1}\mathbf{v}_{1}\mathbf{v}'_{1}\in\mathcal{M}_{k+1}$. By Definition \ref{defi(3.1)}$(e)$, we have
\begin{equation}\label{eq(3.9)}
|\mathbf{v}_{1}|\leq|\mathbf{v}_{1}|+|\mathbf{v}'_{1}|\leq L.
\end{equation}
Combining (\ref{eq(3.7)}), (\ref{eq(3.8)}) and (\ref{eq(3.9)}) show that there exists a constant $c_{3}>0$ such that for any $j\in\{1,\dots,m\}$,
\begin{equation}\label{eq(3.10)}
c_{3}\leq R^{|\mathbf{v}_{j}|}.
\end{equation}
In particular, we can take $c_{3}=r^{L+1}/(C_{1}^{3}C_{4}^{1/n})$. Let $c_{4}:=\lfloor\log c_{3}/\log R\rfloor+1$. Then $|\mathbf{v}_{j}|\leq c_{4}$. Combining these and (\ref{eq(3.2)}) yields
$$\#\{S\in\mathcal{A}_{b}:x\in S(\Omega)\}\leq c_{2}N^{c_{4}},$$
which implies that $\{S_{i}\}_{i=1}^{N}$ satisfies (WSC).
\end{proof}

\section{Hausdorff dimension of self-similar sets}\label{S:4}

In this section, we assume that $M$ is a complete $n$-dimensional smooth orientable Riemannian manifold that is locally Euclidean, i.e., every point of $M$ has a neighborhood which is isometric to an open subset of a Euclidean space.
Let $W\subset M$ be compact, and let $\{S_{i}\}_{i=1}^{N}$ be an IFS satisfying (FTC) of contractive similitudes on some open sets $\Omega\subset W$ with attractor $K\subset W$. Recall that $\rho_i$ denotes the contraction ratio of $S_{i}$. We define
$$\rho:=\min\{\rho_{i}:1\leq i\leq N\}, \quad \rho_{\max}:=\max\{\rho_{i}:1\leq i\leq N\}.$$
Denote the neighborhood types of $\{S_{i}\}_{i=1}^{N}$ by $\{\mathcal{T}_{1},\dots,\mathcal{T}_{q}\}$. Fix a vertex $\omega\in\mathcal{V}_{R}$ such that $[\omega]\in\mathcal{T}_{i}$, where $i\in\{1,\dots,q\}$. Let $\sigma_{1},\dots,\sigma_{m}$ be the offspring of $\omega$ in $\mathcal{G}_{R}$, and let $\mathbf{w}_{k}$ be the unique edge in $\mathcal{G}_{R}$ connecting $\omega$ to $\sigma_{k}$ for $1\leq k\leq m$. Define a \textit{weighted incidence matrix} $A_{\alpha}=(A_{\alpha}(i,j))_{i,j=1}^{q}$ as
$$A_{\alpha}(i,j):=\sum_{k=1}^{m}\{\rho_{\mathbf{w}_{k}}^{\alpha}:
\omega\stackrel{\mathbf{w}_{k}}{\longrightarrow}\sigma_{k},[\sigma_{k}]=\mathcal{T}_{j}\}.$$
We remark that the definition of $A_{\alpha}$ is independent of the choice of $\omega$ above. We denote by $\omega\rightarrow_{R}\sigma$ if $\omega,\sigma\in\mathcal{V}_{R}$ and $\sigma$ is an offspring of $\omega$ in $\mathcal{G}_{R}$. We define an (infinite) \textit{path} in $\mathcal{G}_{R}$ to be an infinite sequence $(\omega_{0},\omega_{1},\dots)$ such that for any $k\geq0$,
$$\omega_{k}\in\mathcal{V}_{k}\quad\text{and}\quad\omega_{k}\rightarrow_{R}\omega_{k+1},$$
where $\omega_{0}=\omega_{{\rm root}}$. Let $\mathbb{P}$ be the set of all paths in $\mathcal{G}_{R}$. If the vertices $\omega_{0}=\omega_{{\rm root}},\omega_{1},\dots,\omega_{k}$ are such that
$$\omega_{j}\rightarrow_{R}\omega_{j+1}\text{ for }1\leq j\leq k-1,$$
then we call the set
$$I_{\omega_{0},\omega_{1},\dots,\omega_{k}}=\{(\sigma_{0},\sigma_{1},\dots)\in\mathbb{P}:\sigma_{j}=\omega_{j}
\text{ for any }0\leq j\leq k\}$$
a \textit{cylinder}. Since the path from $\omega_{0}$ to $\omega_{k}$ is unique in $\mathcal{G}_{R}$, we {\color{blue}}let
$$I_{\omega_{k}}:=I_{\omega_{0},\omega_{1},\dots,\omega_{k}}.$$
For any cylinder $I_{\omega_{k}}$, where $\omega_{k}\in\mathcal{V}_{k}$ and $[\omega_{k}]=\mathcal{T}_{i}$, let
$$\hat{\mu}(\omega_{{\rm root}})=a_{1}=1\quad\text{and}\quad\hat{\mu}(\omega_{k})=\rho_{\omega_{k}}^{\alpha}a_{i},$$
where $[a_{1},\dots,a_{q}]^{T}$ is a $1$-eigenvector of $A_{\alpha}$, normalized so that $a_{1}=1$. We will show that $\hat{\mu}$ is a measure on $\mathbb{P}$ in the following. Note that for two cylinders $I_{\omega}$ and $I_{\omega}'$ with $\omega\in\mathcal{V}_{k},\omega'\in\mathcal{V}_{\ell}$ and $k\leq\ell$, $I_{\omega}\cap I_{\omega}'\neq\emptyset$ iff either $\omega'=\omega$ in the case $k=\ell$ or $\omega'$ is a descendant of $\omega$ in the case $k<\ell$. Whatever, $I_{\omega}'\subset I_{\omega}$. Let $\omega\in\mathcal{V}_{R}$ and $\mathcal{D}:=\{\sigma_{k}\}_{k=1}^{m}$ denote the set of all offspring of $\omega$ in $\mathcal{G}_{R}$. For $k\in\{1,\dots,m\}$, let $\omega\stackrel{\mathbf{w}_{k}}{\longrightarrow}_{R}\sigma_{k}$. Then
$$\begin{aligned}
\sum_{\sigma\in\mathcal{D}}\hat{\mu}(I_{\sigma})
&=\sum_{j=1}^{q}\sum_{\sigma\in\mathcal{D},[\sigma]=\mathcal{T}_{j}}\hat{\mu}(I_{\sigma})\\
&=\sum_{j=1}^{q}\sum_{\sigma\in\mathcal{D},[\sigma]=\mathcal{T}_{j}}\rho_{\sigma}^{\alpha}a_{j}
\\
&=\rho_{\omega}^{\alpha}\sum_{j=1}^{q}\sum_{\sigma\in\mathcal{D},[\sigma]=\mathcal{T}_{j}}
\rho_{\mathbf{w}_{k}}^{\alpha}a_{j}\\
&=\rho_{\omega}^{\alpha}\sum_{j=1}^{q}A_{\alpha}(i,j)a_{j}\\
&=\rho_{\omega}^{\alpha}a_{i}=\hat{\mu}(I_{\omega}).
\end{aligned}$$
Combining these with $\hat{\mu}(\mathbb{P})=\hat{\mu}(\omega_{{\rm root}})=1$ shows that $\hat{\mu}$ is indeed a measure on $\mathbb{P}$. Define $f:\mathbb{P}\longrightarrow W$ by letting $f(\omega_{0},\omega_{1},\dots)$ be the unique point in $\bigcap_{k=0}^{\infty}S_{\omega_{k}}(K)$. It is clear that $f(\mathbb{P})=K$. Let $\widetilde{\mu}:=\hat{\mu}\circ f^{-1}$. Then $\widetilde{\mu}$ is a measure on $K$.

For any bounded Borel set $F\subset M$, let
\begin{equation}\label{eq(4.1)}
\mathcal{B}(F):=\{I_{\omega_{k}}=I_{\omega_{k},\dots,\omega_{k}}:
|S_{\omega_{k}}(\Omega)|\leq|F|<|S_{\omega_{k-1}}(\Omega)|\text{ and }F\cap S_{\omega_{k}}(\Omega)\neq\emptyset\}.
\end{equation}

\begin{lem}\label{lem(4.1)}
There exists a constant $C_{5}>0$, independent of $k$, such that for any bounded Borel set $F\subset M$, we have $\#\mathcal{B}(F)\leq C_{5}$.
\end{lem}
\begin{proof}
Define
$$\begin{aligned}
\widetilde{\mathcal{B}}(F):&=\{\omega_{k}\in\mathcal{V}_{k}:
|S_{\omega_{k}}(\Omega)|\leq|F|<|S_{\omega_{k-1}}(\Omega)|\text{ and }F\cap S_{\omega_{k}}(\Omega)\neq\emptyset\}\\
&=\{\omega_{k}\in\mathcal{V}_{k}:
\rho_{\omega_{k}}\leq|F|/|\Omega|<\rho_{\omega_{k-1}}\text{ and }F\cap S_{\omega_{k}}(\Omega)\neq\emptyset\}.
\end{aligned}$$
Since $I_{\omega_{k}}$ is one-to-one with $\omega_{k}$, we have $\#\mathcal{B}(F)=\#\widetilde{\mathcal{B}}(F)$. Let $b:=|F|/|\Omega|$ and $\omega_{k}\in\widetilde{\mathcal{B}}(F)$. Then there exists a unique $\mathbf{u}\in\mathcal{M}_{k}$ such that $\omega_{k}=(S_{\mathbf{u}},k)$. Let $\mathbf{u}'\preccurlyeq\mathbf{u}$ such that $S_{\mathbf{u}'}\in\mathcal{A}_{b}$. Then
$$\rho_{\mathbf{u}'}\leq b=|F|/|\Omega|<\rho_{\omega_{k-1}}.$$
Thus $\mathbf{u}'\in\mathcal{M}_{k-1}$ or $\mathcal{M}_{k}$. Combining these and Definition \ref{defi(3.1)}$(e)$, we have $|\mathbf{u}|-|\mathbf{u}'|\leq L$.
Note that $F\cap S_{\omega_{k}}(\Omega)\neq\emptyset$ implies that $F\cap S_{\mathbf{u}'}(\Omega)\neq\emptyset$. Since $S_{\mathbf{u}'}\in\mathcal{A}_{b}$, we have
$$|S_{\mathbf{u}'}(\Omega)|=\rho_{\mathbf{u}'}|\Omega|\leq b\Omega.$$
Let $\delta:=2b|\Omega|$ and fix any $x_{0}\in F$. Then
$$S_{\mathbf{u}'}(\Omega)\subset B_{\delta}(x_{0}).$$
Since (FTC) implies (WSC), there exists a constant $\gamma>0$ (independent of $b$) such that for all $x\in U$,
$$\#\{S\in\mathcal{A}_{b}:x\in S(\Omega)\}\leq\gamma.$$
Note that the contraction ratio of $S_{\mathbf{u}}$ is $\rho_{\mathbf{u}}=|\det S_{\mathbf{u}}'(x)|^{\frac{1}{n}}$ for any $x\in W$. Let $A\subset W$ be a measurable set. Then by Lemma \ref{lem(1.1)}, we have
$$\nu(S_{\mathbf{u}}(A))\geq(b\rho)^{n}\nu(A).$$
Combining these we have
$$\begin{aligned}
(b\rho)^{n}\nu(\Omega)\#\{S_{\mathbf{u}'}:F\cap S_{\mathbf{u}'}(\Omega)\neq\emptyset\}&\leq
\sum\{\nu(S_{\mathbf{u}'}(\Omega)):F\cap S_{\mathbf{u}'}(\Omega)\neq\emptyset\}\\
&\leq \gamma\nu(B_{\delta}(x_{0}))\\
&\leq\gamma c_{n}\delta^{n}\quad(\text{by Lemma \ref{lem(1.30)}})\\
&:=\gamma c_{1}b^{n},
\end{aligned}$$
where $c_{n}$ is the volume of the unit ball in $\mathbb{R}^{n}$ and $c_{1}:=c_{n}2^{n}|\Omega|^{n}$. Thus,
$$\#\{S_{\mathbf{u}'}:F\cap S_{\mathbf{u}'}(\Omega)\neq\emptyset\}\leq\frac{\gamma c_{1}}{\rho^{n}\nu(\Omega)}:=c_2.$$
Hence
$$\#\mathcal{B}(F)=\#\widetilde{\mathcal{B}}(F)\leq
N^{L}\#\{S_{\mathbf{u}'}:F\cap S_{\mathbf{u}'}(\Omega)\neq\emptyset\}\leq c_{2}N^{L}.$$
The lemma follows by letting $C_{5}:=c_{2}N^{L}$.
\end{proof}

\begin{proof}[Proof of Theorem \ref{thm(0.5)}]  Use of Lemma \ref{lem(4.1)} and the properties of
measures $\widetilde{\mu}$ on $K$, as in \cite[Theorem 1.2]{Lau-Ngai_2007}; we omit the details.
\end{proof}

\section{Hausdorff dimension of graph self-similar sets}\label{S:41}

In this section, we define graph self-similar sets on Riemannian manifolds, and derive a formula for computing the Hausdorff dimension of such sets. We assume that $M$ is a complete $n$-dimensional Riemannian manifold that is locally Euclidean.

Let $G=(V,E)$ be a graph, where $V=\{1,\dots,t\}$ is the set of vertices and $E$ is the set of all directed edges. We assume that there is at least one edge between two vertices. It is possible that the initial and terminal vertices are same. A \textit{directed path} in $G$ is a finite string $\mathbf{e}=e_{1}\cdots e_{p}$ of edges in $E$ such that the terminal vertex of each $e_{i}$ is the initial vertex of the next edge $e_{i+1}$. For such a path, denote the \textit{length} of $\mathbf{e}$ by $|\mathbf{e}|=p$. For any two vertices $i,j\in V$, and any positive integer $p$, $E^{i,j}$ be the set of all directed edges from $i$ to $j$, let $E_{p}^{i,j}$ be the set of all directed paths of length $p$ from $i$ to $j$, $E_{p}$ be the set of all directed paths of length $p$, $E^{*}$ be the set of all directed paths, i.e.,
$$E_{p}:=\bigcup_{i,j=1}^{p}E_{p}^{i,j}\quad\text{and}\quad E^{*}:=\bigcup_{p=1}^{\infty}E_{p}.$$
For any edge $e\in E$, we assume that there corresponds a contractive similitude $S_{e}$ with ratio $\rho_{e}$ on $M$. For $\mathbf{e}=e_{1}\cdots e_{p}\in E^{*}$, let
$$S_{\mathbf{e}}=S_{e_{1}}\circ\cdots\circ S_{e_{p}}\quad\text{and}\quad \rho_{\mathbf{e}}=\rho_{e_{1}}\cdots \rho_{e_{p}}.$$
Then there exists a unique family of nonempty compact sets $K_{1},\dots,K_{t}$ satisfying
$$K_{i}=\bigcup_{j=1}^{t}\bigcup_{e\in E^{i,j}}S_{e}(K_{j}),\quad i\in\{1,\dots,t\},$$
(see e.g. \cite{Mauldin-Williams_1988,Edgar-Mauldin_1992,Ngai-Wang-Dong_2010,Das-Ngai_2004}).
Define
$$K:=\bigcup_{i=1}^{t}K_{i}.$$
We call $K$ the \textit{graph self-similar set} defined by $G=(V,E)$, and call $G=(V,E)$ the \textit{graph-directed iterated function system} (GIFS) associated with $\{S_{e}\}_{e\in E}$.

Substituting $E^{*}$ for $\Sigma^{*}$ in Definition \ref{defi(3.1)}, we define a sequence of nested index sets $\{\mathcal{F}_{k}\}_{k=1}^{\infty}$ of directed paths. Note that $\{E_{k}\}_{k=1}^{\infty}$ is an example of a sequence of nested index sets of directed paths. Fix a sequence $\{\mathcal{F}_{k}\}_{k=1}^{\infty}$ of nested index sets. For $i,j\in\{1,\cdots,t\}$, we partition $\mathcal{F}_{k}$ to $\mathcal{F}_{k}^{i,j}$ as
$$\mathcal{F}_{k}^{i,j}:=\mathcal{F}_{k}\cap\bigg(\bigcup_{p\geq1}E_{p}^{i,j}\bigg)=\{\mathbf{e}=e_{1}\cdots e_{p}\in \mathcal{F}_{k}:\mathbf{e}\in E_{p}^{i,j}~\text{for some }p\geq1\}.$$
Note that $\mathcal{F}_{k}=\bigcup_{i,j=1}^{t}\mathcal{F}_{k}^{i,j}$. For $i,j\in\{1,\cdots,t\}$, $k\geq1$, let $\mathbb{V}_{k}$ be the set of \textit{$k$-th level vertices} defined as
$$\mathbb{V}_{k}:=\{(S_{\mathbf{e}},i,j,k):\mathbf{e}\in\mathcal{F}_{k}^{i,j},1\leq i,j\leq t\}.$$
For $\mathbf{e}\in\mathcal{F}_{k}^{i,j}$, we call $(S_{\mathbf{e}},i,j,k)$ (or simply $(S_{\mathbf{e}},k)$) a \text{vertex}. For a vertex $\omega=(S_{\mathbf{e}},i,j,k)\in\mathbb{V}_{k}$ with $k\geq1$, let
$$S_{\omega}=S_{\mathbf{e}}\quad\text{and}\quad \rho_{\omega}=\rho_{\mathbf{e}}.$$
Let $\mathcal{F}_{0}=\{1,\dots,t\}$ and $\mathbb{V}_{0}=\{\omega^{1}_{{\rm root}},\dots,\omega^{t}_{{\rm root}}\}$, where $\omega^{i}_{{\rm root}}=(I,i,i,0)$ and $I$ is the identity map on $M$. Then we say $\mathbb{V}_{0}$ is the set of \textit{root vertices}, and $\{\mathcal{F}_{k}\}_{k=0}^{\infty}$ is a sequence of nested index sets if $\{\mathcal{F}_{k}\}_{k=1}^{\infty}$ is. Let $\mathbb{V}:=\bigcup_{k\geq0}\mathbb{V}_{k}$ be the set of all vertices, and $\pi:\bigcup_{k\geq0}\mathcal{F}_{k}\longrightarrow \mathbb{V}$ be defined as
$$\pi(\mathbf{e}):=\begin{cases}(S_{\mathbf{e}},i,j,k),\quad\text{if }\mathbf{e}\in\mathcal{F}_{k}^{i,j}, k\geq1,\\
\omega^{i}_{{\rm root}},\quad\text{if }\mathbf{e}=i\in\mathcal{F}_{0}.\end{cases}$$

Let $\omega\in\mathbb{V}_{k}$ and $\omega'\in\mathbb{V}_{k+1}$. Suppose that there exist directed paths $\mathbf{e}\in\mathcal{F}_{k},\mathbf{e}'\in\mathcal{F}_{k+1}$ and $\mathbf{k}\in E^{*}$ such that $\pi(\mathbf{e})=\omega$, $\pi(\mathbf{e}')=\omega'$ and $\mathbf{e}'=\mathbf{e}\mathbf{k}$. Then we connect a \textit{directed edge} $\mathbf{k}$ from $\omega$ to $\omega'$, and denote this as $\omega\stackrel{\mathbf{k}}{\longrightarrow}\omega'$. We call $\omega$ a \textit{parent} of $\omega'$ and $\omega'$ an \textit{offspring} of $\omega$.
Define a graph $\mathbb{G}:=(\mathbb{V},\mathbb{E})$, where $\mathbb{E}$ is the set of all directed edges of $\mathbb{G}$. Let $\mathbb{G}_{R}:=(\mathbb{V}_{R},\mathbb{E}_{R})$ be the \textit{reduced graph} of $\mathbb{G}$, defined as in Section \ref{S:3} similarly, where $\mathbb{V}_{R}$ and $\mathbb{E}_{R}$ are the sets of all vertices and all directed edges, respectively.

Let $\mathbf{\Omega}=\{\Omega_{i}\}_{i=1}^{t}$, where $\Omega_{i}\subset M$ is a nonempty bounded open set for any $i\in\{1,\dots,t\}$. We say that $\mathbf{\Omega}$ is \textit{invariant} under the GIFS $G=(V,E)$ if
$$\bigcup_{e\in E^{i,j}}S_{e}(\Omega_{j})\subset\Omega_{i},\quad i,j\in\{1,\cdots,t\}.$$
Since $S_{e}$ is a contractive similitude for any $e\in E^{i,j}$, such a family always exists. Fix an invariant family $\mathbf{\Omega}=\{\Omega_{i}\}_{i=1}^{t}$ of $G=(V,E)$. Let $\omega=(S_{\mathbf{e}},i,j,k)\in\mathbb{V}_{k}$ with $\mathbf{e}\in E_{q}^{i,j}$ and $\omega=(S_{\mathbf{e}'},i',j',k)\in\mathbb{V}_{k}$ with $\mathbf{e}'\in E_{s}^{i',j'}$, where $q,s>0$ are integers. We say that two vertices $\omega$, $\omega'$ are \textit{neighbors} (with respect to $\mathbf{\Omega}$) if
$$i=i'\quad\text{and}\quad S_{\mathbf{e}}(\Omega_{j})\cap S_{\mathbf{e}'}(\Omega_{j'})\neq\emptyset.$$
Let
$$\mathcal{N}(\omega):=\{\omega':\omega'\in\mathbb{V}_{k}\text{ is a neighbor of }\omega\},$$
which is called the \textit{neighborhood} of $\omega$ (with respect to $\mathbf{\Omega}$).
\begin{defi}\label{defi(41.1)}
For any two vertices $\omega=(S_{\mathbf{e}_{\omega}},i_{\omega},j_{\omega},k)\in\mathbb{V}_{k}$ and $\omega'=(S_{\mathbf{e}_{\omega'}},i_{\omega'},j_{\omega'},k')\in\mathbb{V}_{k'}$, let $\sigma=(S_{\mathbf{e}_{\sigma}},i_{\omega},j_{\sigma},k)\in\mathcal{N}(\omega)$ and $\sigma'=(S_{\mathbf{e}_{\sigma'}},i_{\omega'},j_{\sigma'},k')\in\mathcal{N}(\omega')$. Assume that
$$\tau=S_{\omega'}S_{\omega}^{-1}:\bigcup_{\sigma\in\mathcal{N}(\omega)}S_{\sigma}(\Omega_{j_{\sigma}})
\longrightarrow \bigcup_{i=1}^{t}\Omega_{i}$$
induces a bijection $f_{\tau}:\mathcal{N}(\omega)\longrightarrow\mathcal{N}(\omega')$ defined by
\begin{equation}\label{eq(41.1)}
f_{\tau}(\sigma)=f_{\tau}(S_{\mathbf{e}_{\sigma}},i_{\omega},j_{\sigma},k)=(\tau\circ S_{\mathbf{e}_{\sigma'}},i_{\omega'},j_{\sigma'},k').
\end{equation}
We say $\omega$ and $\omega'$ are \textit{equivalent}, i.e., $\omega\sim\omega'$, if the following conditions hold:
\begin{itemize}
\item[$(a)$] $\#\mathcal{N}(\omega)=\#\mathcal{N}(\omega')$ and $j_{\sigma}=j_{\sigma'}$ in \eqref{eq(41.1)};
\item[$(b)$] for $\sigma\in\mathcal{N}(\omega)$ and $\sigma'\in\mathcal{N}(\omega')$ such that $f_{\tau}(\sigma)=\sigma'$, and for any positive integer $k_{0}\geq1$, a directed path $\mathbf{e}\in E^{*}$ satisfies $(S_{\sigma}\circ S_{\mathbf{e}},k+k_{0})\in\mathbb{V}_{k+k_{0}}$ if and only if it satisfies
$(S_{\sigma'}\circ S_{\mathbf{e}},k'+k_{0})\in\mathbb{V}_{k'+k_{0}}$.
\end{itemize}
\end{defi}
It is easy to check that $\sim$ is an equivalence relation. Denote the equivalent class of $\omega$ by $[\omega]$, and call it the \textit{neighborhood types} of $\omega$ (with respect to $\mathbf{\Omega}$).

For a graph $\mathbb{G}=(\mathbb{V},\mathbb{E})$, we can prove that $\mathbb{G}$ satisfies Proposition \ref{prop(3.1)} as in \cite[Propositin 2.5]{Ngai-Wang-Dong_2010}; we omit the details. We now define the graph finite type condition.

\begin{defi}\label{defi(41.2)}
Let $M$ be a complete $n$-dimensional Riemannian manifold that is locally Euclidean, let $G=(V,E)$ be a GIFS, and let $\{S_{e}\}_{e\in E}$ be a family of contractive similitudes defined on $M$.
If there exists an invariant family of nonempty bounded open sets $\mathbf{\Omega}=\{\Omega_{i}\}_{i=1}^{t}$ with respect to some sequence of nested index sets $\{\mathcal{F}_{k}\}_{k=0}^{\infty}$ such that
$$\#\mathbb{V}/_{\sim}<\infty,$$
then we say that $G=(V,E)$ satisfies the \textit{graph finite type condition} (GFTC).
We call such an invariant family $\mathbf{\Omega}$ a \textit{graph finite type condition family} of $G$.
\end{defi}

By assuming a GIFS satisfies (GFTC), we get a formula for $\dim_{{\rm H}}(K)$.

\begin{proof}[Proof of Theorem~\ref{thm(41.1)}] The proof is similar to that of \cite[Theorem 1.1]{Ngai-Wang-Dong_2010} and the definition of a weighted incidence matrix is the same as that in Section \ref{S:4}; we omit the details.
\end{proof}

\section{Examples}\label{S:5}
In this section, we construct some examples of CIFSs and GIFSs on Riemannian manifolds satisfying (FTC) and (GFTC), respectively.

Let $\{f_{i}\}_{i=1}^{N}$ be a contractive CIFS on a compact set $W_0\subset\mathbb{R}^{n}$, i.e., $f_{i}$ is $C^{1+\varepsilon}$ where $0<\varepsilon<1$, and there exists an open and connected set $U_{0}\supset W_{0}$ such that for any $i\in\{1,\dots,N\}$, $f_{i}$ can be extended to an injective conformal map $f_{i}:U_{0}\longrightarrow U_{0}$. Let $M$ be a complete $n$-dimensional smooth Riemannian manifold. Assume that there exists a diffeomorphism
$$\varphi:U\longrightarrow U_{0},$$
where $U\subset M$ is open and connected.
Define
\begin{equation}\label{eq(5.1)}
S_{i}:=\varphi^{-1}\circ f_{i}\circ\varphi:U\longrightarrow S_{i}(U)\quad\text{for any }i\in\{1,\dots,N\}.
\end{equation}
By \cite[Proposition 7.1]{Ngai-Xu_2022}, $\{S_{i}\}_{i=1}^{N}$ is a contractive CIFS on $U$.

Fix a sequence of nested index sets $\{\mathcal{M}_{k}\}_{k=0}^{\infty}$. Let $\widetilde{\mathcal{V}}$ and $\mathcal{V}$ be the sets of all vertices with respect to $\{\mathcal{M}_{k}\}_{k=0}^{\infty}$ of $\{f_{i}\}_{i=1}^{N}$ and $\{S_{i}\}_{i=1}^{N}$, respectively.
For $\tilde{\omega}=(f_{\mathbf{u}},k)\in\widetilde{\mathcal{V}}_{k}$ and $\omega=(S_{\mathbf{u}},k)\in\mathcal{V}_{k}$, where $\mathbf{u}\in\mathcal{M}_{k}$, $k\geq0$, we define $f_{\tilde{\omega}}:=f_{\mathbf{u}}$ and $S_{\omega}:=S_{\mathbf{u}}$. It follows from \eqref{eq(5.1)} that
\begin{equation}\label{eq(5.01)}
S_{\omega}=\varphi^{-1}\circ f_{\tilde{\omega}}\circ\varphi.
\end{equation}

\begin{prop}\label{prop(5.1)}
Let $S_{i}$ be defined as in (\ref{eq(5.1)}), where $i\in\{1,\dots,N\}$. If $\{f_{i}\}_{i=1}^{N}$ satisfies (FTC), then $\{S_{i}\}_{i=1}^{N}$ satisfies (FTC).
\end{prop}
\begin{proof}
For any $\tilde{\omega},\tilde{\omega}'\in\widetilde{\mathcal{V}}$ and $\tilde{\omega}\sim\tilde{\omega}'$, by Definition \ref{defi(3.2)}, there exist $\tilde{\sigma}\in\Omega(\tilde{\omega})$ and $\tilde{\sigma}'\in\Omega(\tilde{\omega}')$ such that
\begin{equation}\label{eq(5.2)}
f_{\tilde{\omega}'}^{-1}\circ f_{\tilde{\sigma}'}=f_{\tilde{\omega}}^{-1}\circ f_{\tilde{\sigma}}.
\end{equation}
For any $\omega,\omega'\in\mathcal{V}$, $\sigma\in\Omega(\omega)$ and $\sigma'\in\Omega(\omega')$, we have
$$\begin{aligned}
S_{\omega'}^{-1}\circ S_{\sigma'}&=\varphi^{-1}\circ f_{\tilde{\omega}'}^{-1}\circ\varphi\circ\varphi^{-1}\circ f_{\tilde{\sigma}'}\circ\varphi\quad\text{(by~(\ref{eq(5.01)}))}\\
&=\varphi^{-1}\circ f_{\tilde{\omega}}^{-1}\circ f_{\tilde{\sigma}}\circ\varphi\quad\text{(by~(\ref{eq(5.2)}))}\\
&=\varphi^{-1}\circ f_{\tilde{\omega}}^{-1}\circ\varphi\circ\varphi^{-1}\circ f_{\tilde{\sigma}}\circ\varphi\\
&=S_{\omega}^{-1}\circ S_{\sigma}.
\end{aligned}$$
It follows that $\omega\sim\omega'$. Since $\{f_{i}\}_{i=1}^{N}$ satisfies (FTC),
we have $\#\widetilde{\mathcal{V}}/_{\sim}<\infty$. Thus, $\#\mathcal{V}/_{\sim}<\infty$. This proves the proposition.
\end{proof}

Let
$$\mathbb{S}^{n}:=\left\{(x_{1},\dots,x_{n+1})\in \mathbb{R}^{n+1}:\sum_{i=1}^{n+1}x_{i}^{2}=1\right\},\quad
\mathbb{D}^{n}:=\left\{(x_{1},\dots,x_{n})\in \mathbb{R}^{n}:\sum_{i=1}^{n}x_{i}^{2}<1\right\}.$$
Let $\mathbb{S}^{n}_{+}$ be the upper hemisphere of $\mathbb{S}^{n}$, and  define the stereographic projection $\varphi:\mathbb{S}^{n}_{+}\rightarrow \mathbb{D}^{n}$ as
$$
\varphi(x_{1},\dots,x_{n+1})=\frac{1}{1+x_{n+1}}(x_{1},\dots,x_{n}):=(y_{1},\dots,y_{n}).
$$
Then
$$\varphi^{-1}(y_{1},\dots,y_{n})=\frac{1}{|\boldsymbol{y}|^{2}+1}\big(2y_{1},\dots,2y_{n},1-|\boldsymbol{y}|^{2}\big),$$
where $|\boldsymbol{y}|^{2}=y_{1}^{2}+\cdots+y_{n}^{2}$.

For convenience, we consider the case of $n=2$. We will give some actual examples of Proposition \ref{prop(5.1)}. The first example below satisfies (OSC), while the other two satisfy (FTC) but not (OSC).

\begin{exam}\label{exam(5.1)}
Let $\{f_{i}\}_{i=1}^{3}$ be a Sierpinski gasket on $\mathbb{R}^{2}$, i.e., for $\boldsymbol{x}\in\mathbb{R}^{2}$,
$$f_{1}(\boldsymbol{x})=\frac{1}{2}\boldsymbol{x}+\bigg(0,\frac{1}{2}\bigg),\quad
f_{2}(\boldsymbol{x})=\frac{1}{2}\boldsymbol{x}+\bigg(-\frac{1}{4},0\bigg),\quad
f_{3}(\boldsymbol{x})=\frac{1}{2}\boldsymbol{x}+\bigg(\frac{1}{4},0\bigg).$$
Let $S_{i}$ be defined as in (\ref{eq(5.1)}), where $i\in\{1,2,3\}$. Then $\{S_{i}\}_{i=1}^{3}$ is a CIFS satisfying (OSC) on $\mathbb{S}^{2}_{+}$ (see Figure \ref{fig.1}(a)).
\end{exam}

\begin{exam}\label{exam(5.2)}
Let $\{f_{i}\}_{i=1}^{4}$ be a CIFS defined as in \cite{Lau-Ngai_2007} satisfying (FTC), i.e., for $\boldsymbol{x}\in\mathbb{R}^{2}$,
$$\begin{aligned}
f_{1}(\boldsymbol{x})&=\rho\boldsymbol{x}+\bigg(\frac{1}{2}\rho,0\bigg),\qquad
f_{2}(\boldsymbol{x})=r\boldsymbol{x}+\bigg(\rho-\rho r+\frac{1}{2}r,0\bigg),\\
f_{3}(\boldsymbol{x})&=r\boldsymbol{x}+\bigg(1-\frac{1}{2}r,0\bigg),\qquad
f_{4}(\boldsymbol{x})
=r\boldsymbol{x}+\bigg(\frac{1}{2}r,1-r\bigg),
\end{aligned}$$
where $0<\rho,r<1$ and $\rho+2r-\rho r\leq1$. Let $S_{i}$ be defined as in (\ref{eq(5.1)}), where $i\in\{1,2,3,4\}$. By Proposition \ref{prop(5.1)}, $\{S_{i}\}_{i=1}^{4}$ is a CIFS satisfying (FTC) on $\mathbb{S}^{2}_{+}$ (see Figure \ref{fig.1}(b)).
\end{exam}

\begin{exam}\label{exam(5.3)}
Let $\{f_{i}\}_{i=1}^{3}$ be a golden Sierpinski gasket defined as in \cite{Ngai-Wang_2001} satisfying (FTC), i.e., for $\boldsymbol{x}\in\mathbb{R}^{2}$,
$$
f_{1}(\boldsymbol{x})=\rho \boldsymbol{x},\quad
f_{2}(\boldsymbol{x})=\rho \boldsymbol{x}+\big(\rho^{2},0\big),\quad
f_{3}(\boldsymbol{x})=\rho^{2} \boldsymbol{x}+\big(\rho,\rho\big),
$$
where $\rho=\big(\sqrt{5}-1\big)/2$. Let $S_{i}$ be defined as in (\ref{eq(5.1)}), where $i\in\{1,2,3\}$. By Proposition \ref{prop(5.1)}, $\{S_{i}\}_{i=1}^{3}$ is a CIFS satisfying (FTC) on $\mathbb{S}^{2}_{+}$ (see Figure \ref{fig.1}(c)).
\end{exam}

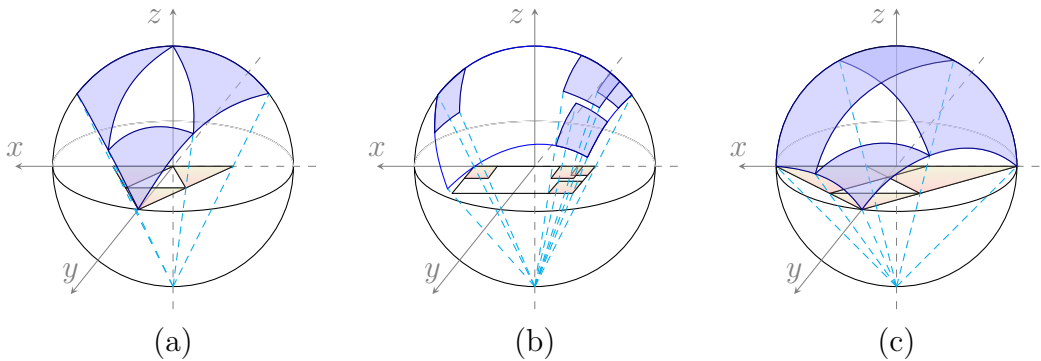
\begin{figure}[htbp]
\begin{center}
   \begin{tikzpicture}[scale=0.2,domain=0:180,>=stealth]
    \coordinate (org) at (0,0);
    \draw (0,0) circle[radius=8];
    \draw (org) ellipse (8cm and 3cm);
    \draw[help lines,dashed](9.5,0)--(0,0);
    \draw[help lines,dashed](0,-9.5)--(0,0);
    \draw[help lines,dashed](5.8,7.25)--(0,0);
    \draw[white] plot ({8*cos(\x)},{3*sin(\x)});
    \foreach\i/\text in{{-6.8,-8.5}/y,{-10.5,0}/x,{0,10.5}/{}}
    \draw[help lines,->] (org)node[above right]{}--(\i)node[above]{$\text$};
    \draw[help lines,->]node[left] at (0,10){$z$};
    \draw[thin,black] (-4,0)--(4,0);
    \draw[thin,black] (-4,0)--(-2.3,-2.875);
    \draw[thin,black] (4,0)--(-2.3,-2.875);
    \draw[thin,black] (-3.15,-1.4375)--(0.85,-1.4375);
    \draw[thin,black] (0,0)--(0.85,-1.4375);
    \draw[thin,black] (-3.15,-1.4375)--(0,0);
    \draw[top color=olive!60!white!30,bottom color=red!30,opacity=.5](-4,0)--(0,0)--(-3.15,-1.4375)--(-4,0);
    \draw[top color=olive!60!white!30,bottom color=red!30,opacity=.5](4,0)--(0,0)--(0.85,-1.4375)--(4,0);
    \draw[top color=olive!60!white!30,bottom color=red!30,opacity=.5](-2.3,-2.875)--(-3.15,-1.4375)--(0.85,-1.4375)--(-2.3,-2.875);
    \draw[thin,cyan,dashed](0,-8)--(-6.4,4.86);
    \draw[thin,cyan,dashed](0,-8)--(-4.4,1.2);
    \draw[thin,cyan,densely dashed](0,-8)--(1.36,2.4);
    \draw[thin,cyan,densely dashed](0,-8)--(6.38,4.8);
    \draw[thin,blue](6.36,4.85) to [out=129.2, in=-2.7](0,8) to [out=182, in=51](-6.39,4.81);
    \draw[thin,blue](-2.3,-2.82) to [out=58, in=-167.8](6.36,4.85);
    \draw[thin,blue](-2.3,-2.875)arc(26:29.58:140);
    \draw[thin,blue](1.35,2.17) to [out=158, in=49](-4.3,1.05);
    \draw[thin,blue](-4.3,1.1) to [out=84, in=-148](0,8);
    \draw[thin,blue](0,8) to [out=-59, in=86](1.35,2.17);
    \draw[fill=blue!30,opacity=.5] (1.35,2.17) to [out=158, in=49](-4.3,1.05) to [out=-60, in=118](-2.31,-2.82) to [out=58, in=-134](1.35,2.17); 
    \draw[fill=blue!30,opacity=.5] (0,8) to [out=182, in=51](-6.39,4.81) to [out=-56, in=121](-4.3,1.1) to [out=84, in=-148](0,8); 
    \draw[fill=blue!30,opacity=.5] (6.36,4.85) to [out=129.2, in=-2.7](0,8) to [out=-59, in=86](1.35,2.17) to [out=43, in=-167.8](6.36,4.85); 
     \foreach \i/\tex in {0/(a)}
     \draw(0,-10)node[below]{\tex};
     \foreach \i/\tex in {0/}
     \draw(0,-19)node[below]{\tex};
   \end{tikzpicture}
   \quad
  \begin{tikzpicture}[scale=.2,domain=0:180,>=stealth]
    \coordinate (org) at (0,0);
    \draw (0,0) circle[radius=8];
    \draw (org) ellipse (8cm and 3cm);
    \draw[help lines,dashed](9.5,0)--(0,0);
    \draw[help lines,dashed](0,-9.5)--(0,0);
    \draw[help lines,dashed](5.8,7.25)--(0,0);
    \draw[white] plot ({8*cos(\x)},{3*sin(\x)});
    \foreach\i/\text in{{-6.8,-8.5}/y,{-10.5,0}/x,{0,10.5}/{}}
    \draw[help lines,->] (org)node[above right]{}--(\i)node[above]{$\text$};
    \draw[help lines,->]node[left] at (0,10){$z$};
    \draw[thin,black] (-4.04,0)--(4.02,0);
    \draw[thin,black] (-5.52,-1.8)--(2.52,-1.8);
    \draw[thin,black] (4.02,0)--(2.5,-1.8);
    \draw[thin,black] (-4.02,0)--(-5.5,-1.8);
    \draw[thin,black] (2.4,-0.6)--(2.85,0);
    \draw[thin,black] (2.4,-0.6)--(3.5,-0.6);
    \draw[thin,black] (1.7,0)--(1.08,-0.8);
    \draw[thin,black] (1.08,-0.8)--(2.58,-0.8);
    \draw[thin,black] (2.58,-0.8)--(3.2,0);
    \draw[thin,black] (-3.1,-0.8)--(-2.5,0);
    \draw[thin,black] (-3.1,-0.8)--(-4.68,-0.8);
    \draw[thin,black] (3.14,-1)--(1.6,-1);
    \draw[thin,black] (1.6,-1)--(0.9,-1.8);
    \draw[top color=olive!60!white!30,bottom color=red!30,opacity=.5]
    (4,0)--(3.5,-0.6)--(2.4,-0.6)--(2.85,0)--(4,0);
    \draw[top color=olive!60!white!30,bottom color=red!30,opacity=.5]
    (3.2,0)--(2.58,-0.8)--(1.08,-0.8)--(1.7,0)--(3.2,0);
     \draw[top color=olive!60!white!30,bottom color=red!30,opacity=.5]
    (-4.7,-0.8)--(-3.1,-0.8)--(-2.5,0)--(-4,0)--(-4.68,-0.8);
     \draw[top color=olive!60!white!30,bottom color=red!30,opacity=.5]
    (1.6,-1)--(3.14,-1)--(2.52,-1.8)--(0.9,-1.8)--(1.6,-1);
   \draw[thin,blue] (-6.45,4.75) to [out=-100, in=110](-5.8,-1.5) to [out=45, in=150] (3.5,0.6) to [out=63, in=-138] (6.45,4.75) to [out=129.1, in=-2](0,8) to [out=181, in=50](-6.45,4.75);
    \draw[thin,cyan,dashed](0,-8)--(-5.8,-1.5);
    \draw[thin,cyan,densely dashed](0,-8)--(-6.45,4.75);
    \draw[thin,cyan,densely dashed](0,-8)--(-5,3.6);
    \draw[thin,cyan,densely dashed](0,-8)--(3.5,0.6);
    \draw[thin,cyan,densely dashed](0,-8)--(6.45,4.75);
    \draw[thin,blue] (-4.6,6.55) to [out=-140, in=52] (-6.45,4.75) to [out=-102, in=89.5] (-6.65,2.25) to [out=46, in=-146] (-5,3.6) to [out=90, in=-105] (-4.6,6.55);
    \draw[fill=blue!30,opacity=.5](-4.6,6.55) to [out=-140, in=52] (-6.45,4.75) to [out=-102, in=89.5] (-6.65,2.25) to [out=46, in=-146] (-5,3.6) to [out=90, in=-105] (-4.6,6.55);

    \draw[thin,cyan,densely dashed](0,-8)--(2.74,4.2);
    \draw[thin,blue] (1.4,1.37) to [out=-11,in=152] (3.5,0.6) to [out=63,in=-126](4.83,2.95) to [out=142,in=-22](2.74,4.2)to[out=-125,in=72](1.4,1.37);
    \draw[fill=blue!30,opacity=.5](1.4,1.37) to [out=-11,in=152] (3.5,0.6) to [out=63,in=-126](4.83,2.95) to [out=142,in=-22](2.74,4.2)to[out=-125,in=72](1.4,1.37);

    \draw[thin,cyan,densely dashed](0,-8)--(4.15,4.95);
    \draw[thin,blue](4.15,4.95) to [out=-32,in=137] (5.565,3.87) to [out=51,in=-135] (6.45,4.75) to [out=127,in=-42] (5.05,6.2) to [out=-132,in=60] (4.15,4.95);
    \draw[fill=blue!30,opacity=.5](4.15,4.95) to [out=-32,in=137] (5.565,3.87) to [out=51,in=-135] (6.45,4.75) to [out=127,in=-42] (5.05,6.2) to [out=-132,in=60] (4.15,4.95);

    \draw[thin,cyan,dashed](0,-8)--(1.9,5);
    \draw[thin,cyan,densely dashed](0,-8)--(4.3,3.9);
    \draw[thin,blue](3.1,7.375) to [out=-125,in=70] (1.9,5) to [out=-20,in=150] (4.3,3.9) to [out=55,in=-130] (5.65,5.65) to [out=138,in=-21] (3.1,7.375);
    \draw[fill=blue!30,opacity=.5](3.1,7.375) to [out=-125,in=70] (1.9,5) to [out=-20,in=150] (4.3,3.9) to [out=55,in=-130] (5.65,5.65) to [out=138,in=-21] (3.1,7.375);
   \foreach \i/\tex in {0/(b)}
     \draw(0,-10)node[below]{\tex};
     \foreach \i/\tex in {0/\text{}}
     \draw(0,-19)node[below]{\tex};
   \end{tikzpicture}
   \quad
  \begin{tikzpicture}[scale=.2,domain=0:180,>=stealth]
    \coordinate (org) at (0,0);
    \draw (0,0) circle[radius=8];
    \draw (org) ellipse (8cm and 3cm);
    \draw[help lines,dashed](9.5,0)--(0,0);
    \draw[help lines,dashed](0,-9.5)--(0,0);
    \draw[help lines,dashed](5.8,7.25)--(0,0);
    \draw[white] plot ({8*cos(\x)},{3*sin(\x)});
    \foreach\i/\text in{{-6.8,-8.5}/y,{-10.5,0}/x,{0,10.5}/{}}
    \draw[help lines,->] (org)node[above right]{}--(\i)node[above]{$\text$};
    \draw[help lines,->]node[left] at (0,10){$z$};
    \draw[thin,black] (-8,0)--(-2.3,-2.875);
    \draw[thin,black] (-2.3,-2.875)--(8,0);
    \draw[thin,black] (8,0)--(-8,0);
    \draw[thin,black] (-4.4375,-1.797)--(1.5625,-1.797);
    \draw[thin,black] (-4.4375,-1.797)--(2,0);
    \draw[thin,black] (-2,0)--(1.5625,-1.797);
    \draw[top color=olive!60!white!30,bottom color=red!30,opacity=.5]
    (1.5625,-1.797)--(-2,0)--(8,0)--(1.5625,-1.797);
    \draw[top color=olive!60!white!30,bottom color=red!30,opacity=.5]
    (-4.4375,-1.797)--(-2.3,-2.875)--(1.5625,-1.797)--(-4.4375,-1.797);
    \draw[top color=olive!60!white!30,bottom color=red!30,opacity=.5](-8,0)--(-4.4375,-1.797)--(2,0)--(-8,0);
    \draw[thin,cyan,densely dashed](0,-8)--(-8,0);
    \draw[thin,cyan,dashed](0,-8)--(-2.3,-2.875);
    \draw[thin,cyan,densely dashed](0,-8)--(8,0);
    \draw[thin,blue] (-8,0) to [out=90, in=-180](0,8) to [out=0, in=90] (8,0) to [out=140, in=55] (-2.3,-2.875)
    to [out=125, in=-5] (-8,0);
    \draw[thin,cyan,densely dashed](0,-8)--(-5.4,-0.51);
    \draw[thin,cyan,densely dashed](0,-8)--(-3.8,7.05);
    \draw[thin,cyan,densely dashed](0,-8)--(3.8,7.05);
    \draw[thin,blue](-5.4,-0.51) to [out=45, in=150] (2.2,0.7);
    \draw[thin,blue](3.8,7.05) to [out=-175, in=75] (-5.4,-0.51);
    \draw[thin,blue](-3.8,7.05) to [out=-20, in=110] (2.2,0.7);
 \draw[fill=blue!30,opacity=.5](-5.4,-0.51) to [out=45, in=150] (2.2,0.7) to [out=-158, in=53] (-2.3,-2.875) to [out=126, in=-19](-5.4,-0.51);
 \draw[fill=blue!30,opacity=.5](8,0) to [out=90, in=0] (0,8) to [out=-180, in=26] (-3.8,7.05) to [out=-20, in=110] (2.2,0.7) to [out=20, in=140] (8,0);
 \draw[fill=blue!30,opacity=.5](-8,0) to [out=90, in=180] (0,8) to [out=0, in=154] (3.8,7.05) to [out=-175, in=75] (-5.4,-0.51) to [out=162, in=-5] (-8,0);
  \foreach \i/\tex in {0/(c)}
     \draw(0,-10)node[below]{\tex};
     \foreach \i/\tex in {0/}
     \draw(0,-19)node[below]{\tex};
  \end{tikzpicture}
  \vspace{-3.2em}
\caption{Figures (a), (b) and (c) are respectively the IFSs in Examples \ref{exam(5.1)}, \ref{exam(5.2)} and \ref{exam(5.3)}.} \label{fig.1}
\end{center}
\end{figure}

Let $M$ be a complete $n$-dimensional smooth Riemannian manifold that is locally Euclidean.
Now, we construct an example of GIFS on $M$ satisfying (GFTC) to illustrate Theorem \ref{thm(41.1)}. Let $G=(V,E)$ be a GIFS of contractive similitudes defined on $\mathbb{R}^{n}$, and $\{O_{i}\}_{i=1}^{t}$, where $O_{i}\subset\mathbb{R}^{n}$, be an invariant family of nonempty bounded open sets under $G=(V,E)$. Let $O:=\bigcup_{i=1}^{t}O_{i}$. For any edge $e\in E$, there corresponds a contractive similitude $f_{e}:O\longrightarrow O$. Assume that there exists a diffeomorphism
$$\varphi:O_{i}\longrightarrow \Omega_{i}\quad\text{for any }i\in\{1,\dots,t\},$$
where $\Omega_{i}\subset M$ is open and connected. Let $\Omega:=\bigcup_{i=1}^{t}\Omega_{i}$.
For any edge $e\in E$, define
\begin{equation}\label{eq(5.3)}
S_{e}:=\varphi^{-1}\circ f_{e}\circ\varphi:\Omega\longrightarrow \Omega.
\end{equation}
As in \cite[Proposition 7.1]{Ngai-Xu_2022}, $\{S_{e}\}_{e\in E}$ is a family of contractive maps on $M$. If for any $e\in E$, $S_{e}$ is a similitude, then $G=(V,E)$, along with $\{\Omega_{i}\}_{i=1}^{t}$ and $\{S_{e}\}_{e\in E}$, forms a GIFS on $M$. The proof of the following proposition is similar to that of Proposition \ref{prop(5.1)}; we omit it.

\begin{prop}\label{prop(5.2)}
Use the above notation and setup. Let $M$ be a complete $n$-dimensional smooth Riemannian manifold that is locally Euclidean. Let $G=(V,E)$ be a GIFS defined on $\mathbb{R}^{n}$ satisfying (GFTC) with $\{O_{i}\}_{i=1}^{t}$ being a GFTC-family and $\{f_{e}\}_{e\in E}$ being an associated family of contractive similitudes.
For any $e\in E$, let $S_{e}$ be a similitude defined as in (\ref{eq(5.3)}). Then the GIFS $G=(V,E)$ defined on $M$ satisfies (GFTC) with $\{\Omega_{i}\}_{i=1}^{t}$ being a GFTC-family and $\{S_{e}\}_{e\in E}$ being an associated family of contractive similitudes. Moreover, for such two GIFSs connected by a diffeomorphism, the neighborhood types, weighted incidence matrices, and the Hausdorff dimension of the corresponding graph self-similar sets are the same.
\end{prop}

Assume that $G=(V,E)$ is a GIFS of contractive similitudes defined on $M$ satisfying (GFTC). Let $\mathcal{T}_{1},\dots,\mathcal{T}_{q}$ be all the distinct neighborhood types with $\mathcal{T}_{i}=[\omega_{{\rm root}}^{i}]$ for any $i\in\{1,\dots,t\}$. Fix a vertex $\omega\in\mathbb{V}_{R}$ such that $[\omega]\in\mathcal{T}_{i}$, where $i\in\{1,\dots,q\}$.
Let $\sigma_{1},\dots,\sigma_{m}$ be the offspring of $\omega$ in $\mathbb{G}_{R}$, let $\mathbf{k}_{\ell}$ be the unique edge in $\mathbb{G}_{R}$ connecting $\omega$ to $\sigma_{\ell}$ for $1\leq \ell\leq m$, and let
$$C_{ij}:=\{\sigma_{\ell}:1\leq \ell\leq m,~[\sigma_{\ell}]=\mathcal{T}_{j}\}.$$
Note that for two edges $\mathbf{k}_{\ell}$ and $\mathbf{k}_{\ell'}$ connecting $\omega$ to two distinct $\sigma_{\ell}$ and $\sigma_{\ell'}$ satisfying $[\sigma_{\ell}]=[\sigma_{\ell}']=\mathcal{T}_{j}$, the contraction ratios $\rho_{\sigma_{\ell}}$ and $\rho_{\sigma_{\ell'}}$ may be different. We can partition $C_{ij}:=C_{ij}(1)\cup\cdots\cup C_{ij}(n_{ij})$ by using $\rho_{\sigma_{\ell}}$, where for $s=1,\dots,n_{ij}$,
$$C_{ij}(s):=\{\sigma_{\ell}\in C_{ij}:\rho_{\sigma_{\ell}}=\rho_{ijs}\},$$
and the $\rho_{ijs}$ are distinct. Thus, for any entry $A_{\alpha}(i,j)$ of the weighted incidence matrix,
$$A_{\alpha}(i,j)=\sum_{s=1}^{n_{ij}}\#C_{ij}(s)\rho_{ijs}^{\alpha}.$$
Moreover, we can write symbolically
$$\mathcal{T}_{i}\longrightarrow\sum_{j=1}^{q}\sum_{s=1}^{n_{ij}}\#C_{ij}(s)\mathcal{T}_{j}(\rho_{ijs}),$$
where the $\mathcal{T}_{j}(\rho_{ijs})$ are defined in an obvious way.
We say that $\mathcal{T}_{i}$ generates $\#C_{ij}(s)$ neighborhoods of type $\mathcal{T}_{j}$ with contraction ratio $\rho_{ijs}$.

For $\boldsymbol{x}\in[0,1]\times[0,1]$, we consider the following iterated function system with overlaps:
$$
h_{1}(\boldsymbol{x})=\frac{1}{2}\boldsymbol{x}+\bigg(0,\frac{1}{4}\bigg),\ \
h_{2}(\boldsymbol{x})=\frac{1}{2}\boldsymbol{x}+\bigg(\frac{1}{4},\frac{1}{4}\bigg),\ \
h_{3}(\boldsymbol{x})=\frac{1}{2}\boldsymbol{x}+\bigg(\frac{1}{2},\frac{1}{4}\bigg),\ \
h_{4}(\boldsymbol{x})=\frac{1}{2}\boldsymbol{x}+\bigg(\frac{1}{4},\frac{3}{4}\bigg).
$$
Iterations of $\{h_i\}_{i=1}^4$ induce iterations on the $2$-torus $\mathbb{T}^2=\mathbb R^2/\mathbb Z^2$, generating an attractor on $\mathbb{T}^2$. We are interested in computing the Hausdorff dimension of the attractor. However, the relations induced by $\{h_i\}_{i=1}^4$ on $\mathbb{T}^2$ are not well-defined functions, making it awkward to apply the theory of IFSs developed in Section~\ref{S:4}. To overcome this difficulty, we will use the GIFS framework and Theorem~\ref{thm(41.1)}, as shown in the following example.

\begin{exam}\label{exam(5.4)}
Let $\{h_i\}_{i=1}^4$ be defined as above, and let $\mathbb{T}^2=\mathbb{S}^{1}\times\mathbb{S}^{1}$ be a $2$-torus, viewed as $[0,1]\times[0,1]$ with opposite sides identified. Let $\mathbb{T}^2$ be endowed with the Riemannian metric induced from $\mathbb{R}^2$. Consider the IFS $\{g_{i}\}_{i=1}^{4}$ on $[0,1]\times[0,1]$ under the Euclidean metric, where for $\boldsymbol{x}\in[0,1]\times[0,1]$,
$$
g_{1}(\boldsymbol{x})=h_1(\boldsymbol{x}),\quad
g_{2}(\boldsymbol{x})=h_2(\boldsymbol{x}),\quad
g_{3}(\boldsymbol{x})=h_3(\boldsymbol{x}),
$$
and
$$g_{4}(\boldsymbol{x})=\begin{cases}\frac{1}{2}\boldsymbol{x}+\big(\frac{1}{4},\frac{3}{4}\big),\quad \boldsymbol{x}\in[0,1]\times[0,\frac{1}{2}],\\ \frac{1}{2}\boldsymbol{x}+\big(\frac{1}{4},-\frac{1}{4}\big),\quad \boldsymbol{x}\in[0,1]\times[\frac{1}{2},1]\end{cases}$$
(see Figure \ref{fig.3}(a)). $\{g_{i}\}_{i=1}^{4}$ induces four relations $\{S_{i}\}_{i=1}^{4}$ on $\mathbb{T}^2$. We are interested in the Hausdorff dimension of the attractor generated by $\{S_{i}\}_{i=1}^{4}$.
Note that the image of $S_4$ is a connected rectangle in $\mathbb{T}^2$, but the image of $g_4$ is divided into two rectangles in $\mathbb{R}^2$. It is easy to see that $\{S_{i}\}_{i=1}^{4}$ are not well-defined functions and are not contractive under the metric of $\mathbb{T}^{2}$. As a result, we need to use the framework of a GIFS.
Consider the GIFS $G=(V,E)$ defined on $\mathbb{R}^{2}$ associated to $\{g_{i}\}_{i=1}^{4}$ with $V=\{1,2\}$ and $E=\{e_{1},\dots,e_{8}\}$, where $\mathbf{O}=\{O_{1},O_{2}\}$ with $O_{1}=(0,1)\times(0,1/2)$ and $O_{2}=(0,1)\times(1/2,1)$ is the invariant family, and
$$e_{1},e_{2},e_{3}\in E^{1,1},\quad e_{4}\in E^{1,2},\quad e_{5},e_{6},e_{7}\in E^{2,2},\quad e_{8}\in E^{2,1}$$
(see Figure \ref{fig.2}). The associated similitudes are defined as
$$f_{e_{1}}=\frac{1}{2}\boldsymbol{x}+\bigg(0,\frac{1}{4}\bigg),~~
f_{e_{2}}=\frac{1}{2}\boldsymbol{x}+\bigg(\frac{1}{4},\frac{1}{4}\bigg),~~
f_{e_{3}}=\frac{1}{2}\boldsymbol{x}+\bigg(\frac{1}{2},\frac{1}{4}\bigg),~~
f_{e_{4}}=\frac{1}{2}\boldsymbol{x}+\bigg(\frac{1}{4},-\frac{1}{2}\bigg),~~
$$
$$f_{e_{5}}=\frac{1}{2}\boldsymbol{x},~~
f_{e_{6}}=\frac{1}{2}\boldsymbol{x}+\bigg(\frac{1}{4},0\bigg),~~
f_{e_{7}}=\frac{1}{2}\boldsymbol{x}+\bigg(\frac{1}{2},0\bigg),~~
f_{e_{8}}=\frac{1}{2}\boldsymbol{x}+\bigg(\frac{1}{4},\frac{3}{4}\bigg).~~
$$
Let $\Omega_{1}$ and $\Omega_{2}$, viewed as $(0,1)\times(0,1/2)$ and $(0,1)\times(1/2,1)$, be the lower and upper pieces of the interior of $\mathbb{T}^{2}$, respectively. Obviously, for $i=1,2$, and any $e\in E$, there exists a diffeomorphism
$\varphi:O_{i}\longrightarrow \Omega_{i}$ such that $S_{e}$, defined as in \eqref{eq(5.3)}, is a contractive similitude.
Let $K$ and $K_{0}$ be the graph self-similar set of $G=(V,E)$ generated by $\{S_{e}\}_{e\in E}$ and $\{f_{e}\}_{e\in E}$, respectively. Then $\dim_{{\rm H}}(K)=\dim_{{\rm H}}(K_{0})=\log(2+\sqrt{2})/\log2=1.77155\dots$.
\end{exam}
\begin{figure}[htbp]
\begin{center}
\tikzset{every picture/.style={line width=0.75pt}} 

\begin{tikzpicture}[x=0.75pt,y=0.75pt,yscale=-1,xscale=1]

\draw    (170.44,123.61) .. controls (191.13,144.03) and (315.33,147.16) .. (351.66,123.43) ;
\draw [shift={(353.78,121.94)}, rotate = 142.77] [fill={rgb, 255:red, 0; green, 0; blue, 0 }  ][line width=0.08]  [draw opacity=0] (10.72,-5.15) -- (0,0) -- (10.72,5.15) -- (7.12,0) -- cycle    ;
\draw  [fill={rgb, 255:red, 0; green, 0; blue, 0 }  ,fill opacity=1 ] (164.83,118.04) .. controls (164.13,119.22) and (162.61,119.61) .. (161.44,118.91) .. controls (160.26,118.21) and (159.88,116.7) .. (160.57,115.52) .. controls (161.27,114.35) and (162.79,113.96) .. (163.96,114.65) .. controls (165.14,115.35) and (165.52,116.87) .. (164.83,118.04) -- cycle ;
\draw  [fill={rgb, 255:red, 0; green, 0; blue, 0 }  ,fill opacity=1 ] (364.66,118.71) .. controls (363.96,119.88) and (362.45,120.27) .. (361.27,119.58) .. controls (360.1,118.88) and (359.71,117.36) .. (360.41,116.19) .. controls (361.1,115.01) and (362.62,114.62) .. (363.79,115.32) .. controls (364.97,116.02) and (365.36,117.54) .. (364.66,118.71) -- cycle ;
\draw    (170.11,112.28) .. controls (196.73,89.81) and (320.37,97.44) .. (350.28,111.92) ;
\draw [shift={(352.78,113.28)}, rotate = 211.87] [fill={rgb, 255:red, 0; green, 0; blue, 0 }  ][line width=0.08]  [draw opacity=0] (10.72,-5.15) -- (0,0) -- (10.72,5.15) -- (7.12,0) -- cycle    ;
\draw    (158.78,109.28) .. controls (92.94,41.46) and (229.06,41.9) .. (167.05,107.49) ;
\draw [shift={(165.11,109.5)}, rotate = 314.63] [fill={rgb, 255:red, 0; green, 0; blue, 0 }  ][line width=0.08]  [draw opacity=0] (10.72,-5.15) -- (0,0) -- (10.72,5.15) -- (7.12,0) -- cycle    ;
\draw    (358.44,109.28) .. controls (292.61,41.46) and (428.08,41.9) .. (366.05,107.49) ;
\draw [shift={(364.11,109.5)}, rotate = 314.63] [fill={rgb, 255:red, 0; green, 0; blue, 0 }  ][line width=0.08]  [draw opacity=0] (10.72,-5.15) -- (0,0) -- (10.72,5.15) -- (7.12,0) -- cycle    ;
\draw    (370.27,114.04) .. controls (438.68,48.82) and (437.64,183.9) .. (372.62,121.28) ;
\draw [shift={(370.63,119.32)}, rotate = 45.15] [fill={rgb, 255:red, 0; green, 0; blue, 0 }  ][line width=0.08]  [draw opacity=0] (10.72,-5.15) -- (0,0) -- (10.72,5.15) -- (7.12,0) -- cycle    ;
\draw    (366.67,124.69) .. controls (432.19,192.81) and (297.1,192.37) .. (359.44,127.08) ;
\draw [shift={(361.39,125.08)}, rotate = 134.9] [fill={rgb, 255:red, 0; green, 0; blue, 0 }  ][line width=0.08]  [draw opacity=0] (10.72,-5.15) -- (0,0) -- (10.72,5.15) -- (7.12,0) -- cycle    ;
\draw    (154.9,120.5) .. controls (86.64,185.87) and (87.37,50.79) .. (152.53,113.26) ;
\draw [shift={(154.52,115.22)}, rotate = 225.02] [fill={rgb, 255:red, 0; green, 0; blue, 0 }  ][line width=0.08]  [draw opacity=0] (10.72,-5.15) -- (0,0) -- (10.72,5.15) -- (7.12,0) -- cycle    ;
\draw    (166.01,124.02) .. controls (231.52,192.15) and (96.44,191.7) .. (158.78,126.41) ;
\draw [shift={(160.73,124.42)}, rotate = 134.9] [fill={rgb, 255:red, 0; green, 0; blue, 0 }  ][line width=0.08]  [draw opacity=0] (10.72,-5.15) -- (0,0) -- (10.72,5.15) -- (7.12,0) -- cycle    ;

\draw (122,65.9) node [anchor=north west][inner sep=0.75pt]    {$e_{1}$};
\draw (82.5,112) node [anchor=north west][inner sep=0.75pt]    {$e_{2}$};
\draw (122,157) node [anchor=north west][inner sep=0.75pt]    {$e_{3}$};
\draw (248,82) node [anchor=north west][inner sep=0.75pt]    {$e_{4}$};
\draw (248,145) node [anchor=north west][inner sep=0.75pt]    {$e_{8}$};
\draw (390,65.9) node [anchor=north west][inner sep=0.75pt]    {$e_{5}$};
\draw (428.5,112) node [anchor=north west][inner sep=0.75pt]    {$e_{6}$};
\draw (390,157) node [anchor=north west][inner sep=0.75pt]    {$e_{7}$};
\draw (158,132) node [anchor=north west][inner sep=0.75pt]    {$\mathbf{1}$};
\draw (358,132) node [anchor=north west][inner sep=0.75pt]    {$\mathbf{2}$};

\end{tikzpicture}
\vspace{-1.1em}
\caption{The GIFS in Examples \ref{exam(5.4)}.} \label{fig.2}
\end{center}
\end{figure}
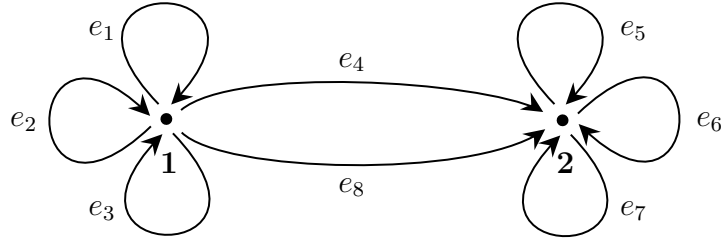
\begin{proof}
For convenience, we write $f_{e_{i}}=f_{i}$ for any $i\in\{1,\dots,8\}$. Let $\mathcal{F}_{k}=E_{k}$ for $k\geq1$, and let $\mathcal{T}_{1}$ and $\mathcal{T}_{2}$ be the neighborhood types of the root neighborhoods $[O_{1}]$ and $[O_{2}]$, respectively. Iterations of the root vertices are shown in Figure \ref{fig.3}(b, c). All neighborhood types are generated after three iterations.
To construct the weighted incidence matrix in the reduced graph $\mathbb{G}_{R}$, we note that
$$\mathbb{V}_{1}=\{(f_{1},1),\dots,(f_{8},1)\}.$$
Denote by $\omega_{1},\dots,\omega_{8}$ the vertices in $\mathbb{V}_{1}$ according to the above order. Then $[\omega_{4}]=\mathcal{T}_{2}$ and $[\omega_{8}]=\mathcal{T}_{1}$. Let $\mathcal{T}_{3}:=[\omega_{1}]$, $\mathcal{T}_{4}:=[\omega_{2}]$, $\mathcal{T}_{5}:=[\omega_{3}]$, $\mathcal{T}_{6}:=[\omega_{5}]$, $\mathcal{T}_{7}:=[\omega_{6}]$, $\mathcal{T}_{8}:=[\omega_{7}]$. Then
$$\mathcal{T}_{1}\longrightarrow\mathcal{T}_{2}+\mathcal{T}_{3}+
\mathcal{T}_{4}+\mathcal{T}_{5}$$
and
$$\mathcal{T}_{2}\longrightarrow\mathcal{T}_{1}+\mathcal{T}_{6}
+\mathcal{T}_{7}+\mathcal{T}_{8},$$
where we write $\mathcal{T}_{i}=\mathcal{T}_{i}(\frac{1}{2})$ for convenience. Since $f_{13}=f_{21}$, the edge $e_{1}e_{3}$ is removed in $\mathbb{G}_{R}$. Hence $\omega_{1}$ has three offspring, i.e.,
$$(f_{11},2),(f_{12},2),(f_{14},2)\in\mathbb{V}_{2},$$
which are of neighborhood types $\mathcal{T}_{3},\mathcal{T}_{4},\mathcal{T}_{2}$, respectively. Iterating $(f_{1},1)$ gives
$$\mathcal{T}_{3}\longrightarrow\mathcal{T}_{2}+\mathcal{T}_{3}+\mathcal{T}_{4}.$$
Since $f_{23}=f_{31}$, the edge $e_{2}e_{3}$ is removed in $\mathbb{G}_{R}$. Hence $\omega_{2}$ has three offspring, i.e.,
$$(f_{21},2),(f_{22},2),(f_{24},2)\in\mathbb{V}_{2}.$$
Note that $[(f_{22},2)]=\mathcal{T}_{4}$ and $[(f_{24},2)]=\mathcal{T}_{2}$. Let $\omega_{9}:=(f_{22},2)$ and $\mathcal{T}_{9}:=[\omega_{9}]$. Then
$$\mathcal{T}_{4}\longrightarrow\mathcal{T}_{2}+\mathcal{T}_{4}+\mathcal{T}_{9}.$$
Note that $\omega_{3}$ has four offspring, i.e.,
$$(f_{31},2),(f_{32},2),(f_{33},2),(f_{34},2)\in\mathbb{V}_{2},$$
with $[(f_{32},2)]=\mathcal{T}_{4}$, $[(f_{33},2)]=\mathcal{T}_{5}$ and $[(f_{34},2)]=\mathcal{T}_{2}$. Let $\omega_{10}:=(f_{31},2)$ and $\mathcal{T}_{10}:=[\omega_{10}]$. Then
$$\mathcal{T}_{5}\longrightarrow\mathcal{T}_{2}+\mathcal{T}_{4}+\mathcal{T}_{5}+\mathcal{T}_{10}.$$
Since $f_{213}=f_{221}$, the edge $e_{2}e_{1}e_{3}$ is removed in $\mathbb{G}_{R}$. Hence $\omega_{9}$ has three offspring, i.e.,
$$(f_{211},2),(f_{212},2),(f_{214},2)\in\mathbb{V}_{3},$$
which are of neighborhood types $\mathcal{T}_{10},\mathcal{T}_{4},\mathcal{T}_{2}$, respectively. Iterating $(f_{22},2)$ gives
$$\mathcal{T}_{9}\longrightarrow\mathcal{T}_{2}+\mathcal{T}_{4}+\mathcal{T}_{10}.$$
Since $f_{313}=f_{321}$, the edge $e_{3}e_{1}e_{3}$ is removed in $\mathbb{G}_{R}$. Hence $\omega_{10}$ has three offspring, i.e.,
$$(f_{311},2),(f_{312},2),(f_{314},2)\in\mathbb{V}_{3},$$
which are of neighborhood types $\mathcal{T}_{10},\mathcal{T}_{4},\mathcal{T}_{2}$, respectively. Iterating $(f_{31},2)$ gives
$$\mathcal{T}_{10}\longrightarrow\mathcal{T}_{2}+\mathcal{T}_{4}+\mathcal{T}_{10}.$$
Let $\mathcal{T}_{11}:=[(f_{65},2)]$ and $\mathcal{T}_{12}:=[(f_{75},2)]$. Using the same argument, it can be checked directly that
$$\mathcal{T}_{6}\longrightarrow\mathcal{T}_{1}+\mathcal{T}_{6}+\mathcal{T}_{7},~~
\mathcal{T}_{7}\longrightarrow\mathcal{T}_{1}+\mathcal{T}_{7}+\mathcal{T}_{11},$$
$$\mathcal{T}_{8}\longrightarrow\mathcal{T}_{1}+\mathcal{T}_{7}+\mathcal{T}_{8}+\mathcal{T}_{12},~~
\mathcal{T}_{11}\longrightarrow\mathcal{T}_{1}+\mathcal{T}_{7}+\mathcal{T}_{12},~~
\mathcal{T}_{12}\longrightarrow\mathcal{T}_{1}+\mathcal{T}_{7}+\mathcal{T}_{12}.$$
Hence the weighted incidence matrix is
$$
A_{\alpha}=\Big(\frac{1}{2}\Big)^{\alpha}\left({\begin{array}{cccccccccccc}
0&1&1&1&1&0&0&0&0&0&0&0\\
1&0&0&0&0&1&1&1&0&0&0&0\\
0&1&1&1&0&0&0&0&0&0&0&0\\
0&1&0&1&0&0&0&0&1&0&0&0\\
0&1&0&1&1&0&0&0&0&1&0&0\\
1&0&0&0&0&1&1&0&0&0&0&0\\
1&0&0&0&0&0&1&0&0&0&1&0\\
1&0&0&0&0&0&1&1&0&0&0&1\\
0&1&0&1&0&0&0&0&0&1&0&0\\
0&1&0&1&0&0&0&0&0&1&0&0\\
1&0&0&0&0&0&1&0&0&0&0&1\\
1&0&0&0&0&0&1&0&0&0&0&1
\end{array}}\right)=:\Big(\frac{1}{2}\Big)^{\alpha}\widetilde{A}_{\alpha},
$$
and the maximal eigenvalue of $\widetilde{A}_{\alpha}$ is $2+\sqrt{2}$.
The GIFS $G=(V,E)$ defined on $\mathbb{R}^{2}$ satisfies (GFTC) with $\{O_{i}\}_{i=1}^{2}$ being a GFTC-family and $\{f_{e}\}_{e\in E}$ being an associated family of contractive similitudes.
By Proposition \ref{prop(5.2)}, the GIFS $G=(V,E)$ defined on $\mathbb{T}^{2}$ satisfies (GFTC) with $\{\Omega_{i}\}_{i=1}^{2}$ being a GFTC-family and $\{S_{e}\}_{e\in E}$ being an associated family of contractive similitudes.
By Theorem \ref{thm(41.1)}, $\dim_{{\rm H}}(K)=\dim_{{\rm H}}(K_{0})=\log(2+\sqrt{2})/\log2=1.77155\dots$.
\end{proof}

\begin{figure}[htbp]
\centering
\subfigure[]
{\includegraphics[width=4.8cm]{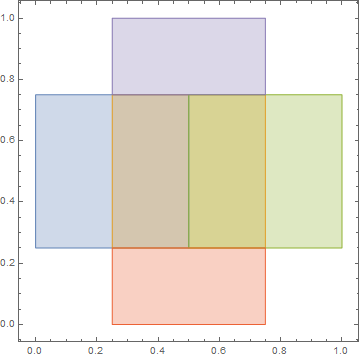}\label{fig.(a)}}
\qquad\qquad
\subfigure[]
{\includegraphics[width=4.8cm]{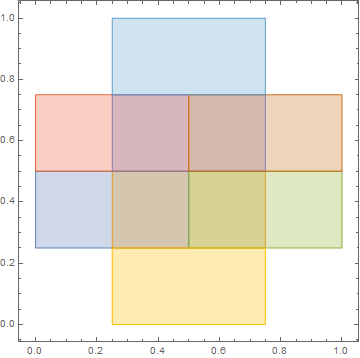}\label{fig.(b)}}
\\
\subfigure[]
{\includegraphics[width=4.8cm]{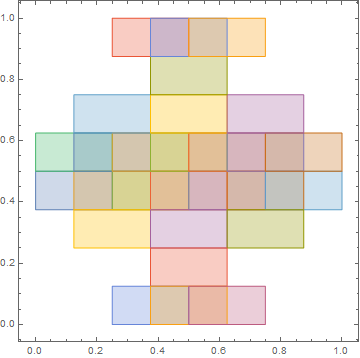}\label{fig.(c)}}
\qquad\qquad
\subfigure[]
{\includegraphics[width=4.8cm]{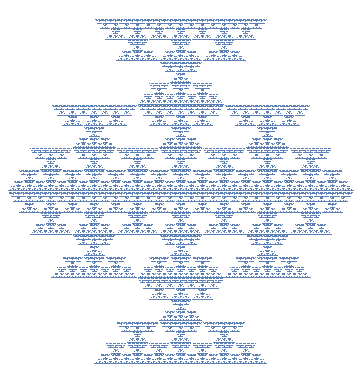}\label{fig.(d)}}

\caption{(a) $\{g_{i}\}_{i=1}^{4}$. (b) The first iteration of $G=(V,E)$ on $\mathbb{R}^{2}$ associated to $\{g_{i}\}_{i=1}^{4}$. (c) The second iteration. (d) The corresponding graph self-similar set.}\label{fig.3}
\end{figure}

\begin{rema}\label{rema(5.1)}
Unlike $\mathbb{T}^{2}$, the family of contractive similitudes associated to a GIFS defined on $\mathbb{R}^{2}$ can be characterized explicitly. In fact, Proposition \ref{prop(5.2)} is not needed in the proof of Example \ref{exam(5.4)}. We may use Theorem \ref{thm(41.1)} and a similar arguement as in the proof of Example \ref{exam(5.4)} to obtain the same result.
\end{rema}

\noindent\textbf{Acknowledgements} The authors are supported in part by the National Natural Science Foundation of China, grant 11771136, and Construct Program of the Key Discipline in Hunan Province. The first author is also supported in part by a Faculty Research Scholarly Pursuit Funding from Georgia Southern University.\\

\noindent\textbf{Declaration}\\

\noindent\textbf{Competing interests} The authors hereby declare that there is no conflict of interest regarding this work.

\end{document}